\documentclass[11pt,leqno]{amsart}

\usepackage{latexsym}
\usepackage{amssymb}
\usepackage{amsmath}
\usepackage{amsthm}
\usepackage{verbatim}
\usepackage{pdfsync}
\usepackage[curve,matrix,arrow,frame,tips]{xy}
\numberwithin{equation}{subsection}

\begin{document}

%%% These commands yield the numbering conventions used in the Kyoto  notes.
%\renewcommand{\thesection}{\Roman{section}}
%\renewcommand{\thesubsection}{\Roman{section}.\arabic{subsection}}

\input KRY.macros

\renewcommand{\E}{{\mathbb E}}

\newcommand{\OO}{\text{\rm O}}
\newcommand{\UU}{\text{\rm U}}

\newcommand{\OK}{O_{\smallkay}}
\newcommand{\DI}{\mathcal D^{-1}}

\newcommand{\pre}{\text{\rm pre}}

\newcommand{\Bor}{\text{\rm Bor}}
\newcommand{\Rel}{\text{\rm Rel}}
\newcommand{\rel}{\text{\rm rel}}
\newcommand{\Res}{\text{\rm Res}}
\newcommand{\TG}{\widetilde{G}}

\parindent=0pt
\parskip=10pt
\baselineskip=14pt

\newcommand{\PP}{\mathcal P}
\renewcommand{\OO}{\mathcal O}
\newcommand{\BB}{\mathbb B}
\newcommand{\GU}{\text{\rm GU}}
\newcommand{\Herm}{\text{\rm Herm}}

\newcommand{\FF}{\mathbb F}
\renewcommand{\MM}{\mathbb M}
\newcommand{\YY}{\mathbb Y}
\newcommand{\VV}{\mathbb V}
\newcommand{\LL}{\mathbb L}

\newcommand{\subover}[1]{\overset{#1}{\subset}}
\newcommand{\supover}[1]{\overset{#1}{\supset}}

\newcommand{\hgs}[2]{\{#1,#2\}}

\newcommand{\wh}[1]{\widehat{#1}}

\newcommand{\Ker}{\text{\rm Ker}}
\newcommand{\YYbar}{\overline{\YY}}
\newcommand{\MMbar}{\overline{\MM}}
\newcommand{\oY}{\overline{Y}}
\newcommand{\dra}{\dashrightarrow}
\newcommand{\dlra}{\longdashrightarrow}

\newcommand{\yy}{\text{\bf y}}

\newcommand{\red}{\text{\rm red}}

\newcommand{\inc}{\text{\rm inc}}

\newcommand{\OKs}{O_{\smallkay,s}}
\newcommand{\OKr}{O_{\smallkay,r}}
\newcommand{\Xs}{X^{(s)}}
\newcommand{\Xo}{X^{(0)}}

\newcommand{\marbull}{\marginpar{$\bullet$}}

\newcommand{\ED}{\end{document}}

\newtheorem{example}[theo]{Example}

\title[Special cycles on unitary Shimura varieties I]{Special cycles on unitary Shimura varieties\\ I. Unramified local theory}

\author{Stephen Kudla}
\address{Department of Mathematics\\
 University of Toronto\\
40 St.~George St., Toronto, Ontario M5S 2E4, 
Canada}
\email{skudla@math.toronto.edu}

\author{Michael Rapoport}
\address{Mathematisches Institut der Universit\"at Bonn\\  
Endenicher Allee 60\\
53115 Bonn\\
Germany}
\email{rapoport@math.uni-bonn.de}

%\author{Stephen Kudla
%\medskip\\
%and
%\medskip\\
%Michael Rapoport}

%\address{
%Department of Mathematics, University of Toronto\hfill
%%40 St. George St.
%%Toronto, Ontario M5S 2E4
%%Canada
%%skudla@math.toronto.edu
%\medskip
%\break
%Department of Mathematics, University of WisconsinBonn
%}

\maketitle

\centerline{\bf Abstract}
The supersingular locus in the fiber at $p$ of a Shimura variety attached to a unitary 
similitude group $\text{\rm GU}(1,n-1)$ over $\Q$ is uniformized by a formal scheme $\Cal N$.
In the case when $p$ is an inert prime, we define special cycles $\ZZ(\xx)$ in 
$\Cal N$, associated to  collections $\xx$ of $m$ `special homomorphisms'
with fundamental matrix $T\in \Herm_m(\OK)$. When $m=n$ and $T$ is nonsingular, 
we show that the cycle $\ZZ(\xx)$ is either empty or is a union of components of the Ekedahl-Oort stratification, 
and we give a necessary and sufficient condition, in terms of $T$, for 
$\ZZ(\xx)$ to be irreducible. When $\ZZ(\xx)$ is zero dimensional -- in which case it reduces to a single point -- 
we determine the length of the corresponding local ring by using a variant 
of the theory of quasi-canonical liftings. We show that this length coincides with the 
derivative of a representation density for hermitian forms.

\section{Introduction}

A relation between a  
generating series constructed from 
arithmetic  
cycles on an integral model of a Shimura curve and the derivative of a Siegel  
Eisenstein series of genus $2$ was  established by one of us in \cite{kudlaannals}. There, the hope is  
expressed that such a relation should hold 
in greater generality for  
integral models of Shimura varieties attached to orthogonal groups of  
signature $(2, n- 2)$ for any $n$.  The case of Shimura curves corresponds to  
$n=3$; the relation for $n=2$ is established in \cite{KRYtiny}, and  
the cases $n=4$ and $n=5$ are considered in \cite{KR1, KR2}.    
However, the case of arbitrary $n$ seems out of reach at the present time, since these Shimura varieties do not represent a moduli problem  
of abelian varieties with additional structure of PEL-type, so that it is  difficult to define and study integral models of them. For the low  
values of $n$ mentioned above, the analysis depends on exceptional isomorphisms
between orthogonal and symplectic groups.

In the present series of papers, we take up an idea already mentioned  
in a brief remark at the very end of section 16 of \cite 
{kudlaannals} and consider integral models of Shimura varieties  
attached to unitary similitude groups of signature $(1, n-1)$  
over $\Q$. These varieties are moduli spaces for abelian varieties for arbitrary $n$, 
and there are good integral models for them, 
at least away from primes of (very) bad reduction, \cite{kottwitzJAMS}, \cite{pappas}, \cite{PR.I}, \cite{PR.II}, 
\cite{PR.III}. 
In a sequel to the present paper, we define special arithmetic cycles in a modular way and study the 
classes determined by such cycles, together with suitable Green currents,  in the arithmetic Chow groups. The 
ultimate goal is to relate the 
generating series defined by the height pairings, or arithmetic intersection numbers, of such classes
to special values of derivatives of Eisenstein series for the group ${\rm U}(n, n)$. 
It should be noted that the complex points of our cycles coincide with the cycles 
defined in \cite{kudlaU21}, \cite{kudlamillson}, and \cite{KMihes}, where the modular properties of 
the generating functions for the associated cohomology classes are established. 

In the present paper, as a first step in this study, we consider the local analogue, in  
which the Shimura variety is replaced by a formal moduli space  
of $p$-divisible groups, the special arithmetic cycles are replaced by formal  
subvarieties, and the special values of the derivative of the  
Eisenstein series are replaced by the derivatives of representation densities of  
hermitian forms.  The link between this local situation and the global one 
is provided by the uniformization of the supersingular 
locus by the formal schemes introduced in  \cite{RZ}.
A similar local analogue occurs in our earlier work  
\cite{KRinventiones}, where the intersection numbers of formal arithmetic cycles on the Drinfeld 
upper half-space, which uniformizes the fibers of bad reduction of Shimura 
curves, are related to the derivative of representation densities of quadratic forms. 
The results of \cite{KRinventiones} are an essential ingredient in the global 
theory of arithmetic cycles on Shimura curves established in \cite{KRYbook}. 
The results of the present paper will play a similar role in the unitary case. 

It turns out that when two global cycles are disjoint on the generic fiber, their intersections 
are supported in the fibers at non-split primes. Thus, for primes not dividing the level, 
there are two cases, depending on whether the 
prime is inert  or ramified in the imaginary quadratic field.  
As indicated in the title,  in the present paper we handle the unramified primes. 

%Note that we consider hermitian forms of signature $(1,n-1)$ for arbitrary $n$, whereas 
%\cite{krinventiones} and \cite{KRY} only involve quadratic forms of signature $(2,1)$. 

We now  describe our results in more detail.

Let $\kay=\Q_{p^2}$ be the unramified quadratic extension of $\Q_p$
and let $\OK=\Z_{p^2}$ be its ring of integers. Let $\FF = \bar \F_p$ and write 
$W=W(\FF)$ for its ring of Witt vectors.
There are two embeddings $\ph_0$ and $\ph_1=\ph_0\circ \s$ of $\kay$  
into $W_\Q=W\tt_\Z\Q$.  Let $\text{\rm Nilp}_W$ be the category of $W$-schemes $S$ 
on which $p$ is locally nilpotent. For any scheme $S$ over $W$,  let
$\bar S=S\times_{\Spec W}\Spec \FF$ be its special fiber. 

The formal scheme on which we work is defined as follows. 
We consider $p$-divisible groups $X$ of dimension $n$ and height  
$2n$ over $W$-schemes $S$,
with an action $\iota:\OK \rightarrow \End(X)$ satisfying the {\it  
signature condition $(1, n-1)$},
\begin{equation}\label{detcond}
\text{\rm char}(\iota(\a),\Lie\, X)(T) = (T-\ph_0(\a))(T-\ph_1(\a))^ 
{n-1}\in \mathcal O_S[T],
\end{equation}
and equipped with  a $p$-principal polarization $\l_{X}$, for which  
the Rosati involution $*$  satisfies $\iota(\a)^*=\iota(\a^\s)$.

Up to isogeny, there is a unique such a triple  $(\X,\iota, \l_\X)$ over $\FF$ 
such that  $\X$ is {\it  
supersingular}\,\footnote{Recall that this means that $\X$ is isogenous to the $n$th power 
of the $p$-divisible group of a supersingular elliptic curve.}.   
Fixing $(\X,\iota, \l_\X)$, we denote by $\Cal N$ the formal scheme over $W$ which  
parametrizes the quadruples $(X, \iota, \l_X,\rho_X)$ over schemes $S$ in $\text{\rm Nilp}_W$,  
where $(X, \iota, \l_X)$  
is as above, and where 
$$\rho_X : X\times_S\bar S\to  \X\times_\FF \bar S$$ 
is a quasi-isogeny which respects the auxilliary structures imposed.   
(See section \ref{DefofMod} and \cite{RZ}, section 1,  for the precise definition of $\Cal N$.)
Then $\Cal N$ is formally  
smooth of relative dimension $n-1$ over $W$, and the underlying reduced  
scheme $\Cal N_{\rm red}$ is a singular scheme of dimension $[(n-1)/2]$ over $\FF$.

In order to explain our results, we need to recall some of the  
results on the structure of $\Cal N_{\rm red}$ due to Vollaard \cite 
{vollaard}, as completed by Vollaard and Wedhorn \cite{vollwedhorn}.   
To the polarized isocrystal $N$ of $\X$  there is associated a  
hermitian vector space $(C, \{\ ,\ \})$ of dimension $n$ over $\kay$  
satisfying the parity condition 
$$\ord \det(C)\equiv n+1 \ {\rm mod}\  2.$$
Here $\det(C) \in \Q_p^\times/N(\kay^\times)$ is the coset determined by $\det((\{c_i,c_j\}))$ for 
any $\kay$ basis $\{c_i\}$ for $C$. 
Note that this condition  
determines $(C, \{\ ,\ \})$  up to isomorphism.  A {\it vertex of  
level} $i$ is an $\OK$-lattice $\L$ in $C$ with
$$p^{i+1}\L^\vee\subset \L\subset p^i\L^\vee,$$
where 
$$\L^\vee = \{\ x\in C\mid \{x,\L\} \subset \OK\ \}$$
is the dual lattice. Such lattices correspond to the vertices of the building $\Cal B(\text{\rm U}(C))$
of the unitary group $U(C)$, hence the terminology. 
The {\it type} of a vertex $\L$  is the index $t(\L)$ of $p^{i+1}\L^ 
\vee$ in $\L$. In fact, $t(\L)$ is always an odd integer between $1$ and $n$.  
To every vertex $\L$ of level $i$, Vollaard and Wedhorn associate a locally  
closed irreducible subset $\Cal V(\L)^o$ of $\Cal N_{\rm red}$ of dimension $\frac{1}{2}(t 
(\L)-1)$  with the following properties: \hfb
\vskip -14pt
a) The closure $\Cal V(\L)$ of $\Cal V(\L)^o$ is the  
finite disjoint union
\begin{equation*}
\Cal V(\L)= \bigcup_{\L'\subset \L} \Cal V(\L')^o,
\end{equation*}
where $\L'$ runs over all vertex lattices of level
$i$  contained in $\L$. Note that for such vertices $t(\L')\leq t(\L)$.\hfb
\vskip -14pt
b) The union of $\Cal V(\L)^o$, as $\L$ ranges over all  
vertices of level $i$, is a connected component $\Cal N_i$ of $\Cal N_{\rm red}$,  
and as $i$ varies, all connected components of $\Cal N_{\rm red}$  
arise in this way.

Thus the combinatorics of the stratification of $\Cal N_{\rm red}$ are controlled by the 
building $\Cal B(\text{\rm U}(C))$, just as the (much simpler) stratification of the special fiber of the formal model of the Drinfeld 
half-space is controlled by the tree $\Cal B(\GL_2(\Q_p))$. 

We next define special cycles on $\Cal N$. Let  $(\YY,\iota,\l_{\YY})$ be the basic object over $\FF$ used in the 
definition of the signature $(1,n-1)$ moduli space $\Cal N$ in the case $n=1$.
Thus  $\YY$ is a supersingular $p$-divisible group over $\FF$ of  
dimension $1$ with $\OK$-action $\iota$ which satisfies the signature  
condition $(1,0)$ with its natural  $p$-principal polarization $\lambda_{\YY}$.
Next, let 
$(\YYbar, \iota,\l_{\YYbar})$ be the triple obtained from  $(\YY,\iota,\l_{\YY})$ by changing $\iota$ to $\iota\circ \s$. 
The $\OK$-action on $\YYbar$  satisfies the signature  
condition $(0,1)$, and,   
again, the triple $(\YYbar, \iota, \l_{\YYbar})$ is unique up to isogeny.  Since 
$\YYbar$ has height $2$, 
the pair $(\YYbar, \iota)$ has a canonical lift $(\bar Y, \iota)$ over $W$, \cite{grossqc}. 

The {\it space of  special  
homomorphisms} is the $\kay$-vector space
$$\VV:=\Hom_{\OK}(\YYbar,\X)\tt_\Z\Q.$$
with $\kay$-valued hermitian form given by
$$
h(x,y) = \l_{\YYbar}^{-1}\circ \hat{y}\circ \l_{\X}\circ x \in \End_ 
{\OK}(\YYbar)\tt\Q \overset{\iota^{-1}}{\isoarrow} \kay.
$$
Here $\hat y$ is the dual of $y$. 
The parity of  $\ord\det(\VV)$ is always odd.

For a pair of integers $(i,j)$ and a special homomorphism $x\in \VV$,  there is an associated {\it special cycle}
$\ZZ_{i,j}(x)$ where, for any $S\in \text{\rm Nilp}_W$,  $\ZZ_{i,j}(x)(S)$ is the subset of points $(X, \iota, \lambda_X, \rho_X)$ in $\Cal N_j(S)$  
where the composition, a quasi-homomorphism,  
$$  \YYbar \times_\FF \bar S\overset{p^i\cdot x} {\lra} \X\times_\FF \bar S\overset{\rho_X^{-1}}{\lra} X\times_S\bar S$$
extends to an $\OK$-linear homomorphism 
$$\bar Y\times_W S \lra X$$
from the canonical lift of $\YYbar$ to $X$.  
Then $\ZZ_{i,j}(x)$ is a relative (formal) divisor  
in $\Cal N_j$. More generally, for an $m$-tuple $\xx = [x_1,\dots,  
x_m]$ of special homomorphisms $x_r\in \VV$,
the {\it associated special cycle} $\ZZ_{i,j}(\xx)$ is the  
intersection of the special cycles
associated to the components of $\xx$. 

For a collection $\xx$ of special homomorphisms,  we define the {\it  
fundamental matrix}
$$T(\xx)=h(\xx, \xx)= (h(x_r,x_s))\in {\rm Herm}_m(\kay).$$ 
If a pair $(i,j)$ is fixed, we let $\widetilde T = p^{2i-j}\,T(\xx)$.

We now assume that $m=n$. 
This will be the situation that arises from the global setting when one considers the arithmetic intersections
of cycles in complementary dimensions.  
First, we show that, if $\widetilde T\notin \Herm_n(\OK)$, then $\ZZ_{i,j}(\bold x)$ is empty. 
Next, assume that  $\det T(\xx) \ne 0$, as will be the case in the global setting
when the cycles do not meet on the generic fiber. When $\widetilde T$ is integral,  our main result describes the structure of the cycle $\ZZ_{i,j}(\xx)$ 
in terms of the Jordan block structure of $\widetilde T$.  

\begin{theo}\label{mainintrothm}
Suppose that $\det T(\xx) \ne 0$ and that $\widetilde T\in \Herm_n(\OK)$. 
Write $\red(\widetilde T)$ for the image of $\widetilde T$ in $\Herm_n(\F_{p^2})$. 

\noindent (i) {\rm (compatibility with the stratification)} $\ZZ_{i,j} 
(\bold x)_{\rm red}$ is a union of strata $\Cal V(\L)^o$ where $\L$ ranges over a 
finite set of vertices of level $j$
which can be explicitly described in terms of $\xx$. 

\noindent  (ii) {\rm (dimension)} Let $r^0(\widetilde T) = n-\text{\rm rank}(\red(\widetilde T))$ be the 
dimension of the radical of the hermitian form $\red(\widetilde T)$. 
Then $\ZZ_{i,j}(\bold x)_{\red}$ is  
purely of dimension 
$\left[(r^0(\widetilde T)-1)/2\right].$

\noindent (iii) {\rm (irreducibility)} Let\footnote{See Theorem~\ref{propirred} for the notation.}
$$\tilde T \simeq \diag(1_{n_0},  p 1_{n_1},  \dots,  p^k 1_{n_k})$$
be a Jordan decomposition of   $\tilde T$ and let
$$n^+_{\text{\rm \ even}} = \sum_{\substack{i \ge 2\\  \text{\rm \  
even}}} n_i\qquad\text{
and}\qquad
n^+_{\text{\rm \ odd}} = \sum_{\substack{i \ge 3\\  \text{\rm \  
odd}}} n_i.$$
Then  $\ZZ_{i,j}(\bold x)_{\rm red}$ is irreducible if and only if
$$\max(n^+_{\text{\rm \ even}}, n^+_{\text{\rm \ odd}}) \le 1.$$
Moreover its dimension is then
$\frac12(n_1+n_{\rm odd}^+-1)$.

\noindent (iv) {\rm (zero-dimensional case)} $\ZZ_{i,j}(\bold x)_{\rm  
red}$ is of dimension zero if and only if  $\tilde T$ is $\OK$-equivalent to $\diag (1_{n-2}, p^a, p^b)$, where $0\leq a<b$ and  
where $a+b$ is odd. In this case, $\ZZ_{i,j}(\bold x)_{\rm red}$  
consists of a single point $\xi$, and  the length of the local ring $ 
\Cal O_{\ZZ_{i,j}(\bold x), \xi}$ is finite and given by
$${\rm length}_W(\Cal O_{\ZZ_{i,j}(\bold x),  \xi})= \frac12\,\sum_{l  
=0}^{a} p^l (a+b+1-2l)
\ .$$

\end{theo}

The right hand side of the last identity can be expressed in terms of representation densities of hermitian forms. 
Recall that, for nonsingular hermitian matrices $S\in \Herm_m(\OK)$ and 
$T \in \Herm_n(\OK)$, with $m\ge n$,
the representation density $\a_p(S,T)$ is defined as 
 \begin{equation*}
\a_p(S,T) = \lim_{k\to\infty} (p^{-k})^{n(2m-n))} |A_{p^k}(S,T)|,
\end{equation*}
 where
$$A_{p^k}(S,T) =\{\ x\in M_{m,n}(\OK/p^k\OK)\mid S[x]\equiv T\mod p^k \ \},$$
with $S[x] = {}^tx S\s(x)$.
The density depends only on the $\GL_m(\OK)$-equivalence class of $S$ 
(resp. the  $\GL_n(\OK)$-equivalence class of $T$).
An explicit formula for $\a_p(S,T)$ has been given by Hironaka, \cite 
{hironaka}. For $r\ge 0$, let $S_r = \diag(S,1_{r})$. Then
$$\a_p(S_r,T) = F_p(S,T; (-p)^{-r})$$
for a polynomial $F_p(S,T;X)\in \Q[X]$, as can be seen immediately
from Hironaka's formula.

If $m=n$ and 
 $\ord(\det(S))+\ord(\det(T))$ is odd, then
$\a_p(S,T)=0$. In this case, we define the derivative of the  
representation density
$$\a'_p(S,T) = -\frac{\d}{\d X}F_p(S,T;X)\vert_{X=1}.$$

The right hand side of the identity in (iv) of Theorem \ref{mainintrothm} is now expressed
in terms of hermitian representation densities, as follows. 
\begin{prop}\label{dender}
Let $S= 1_n$ and $T=\diag(1_{n-2}, p^a, p^b)$ for $0\le a< b$ with  
$a+b$ odd.
Then $\a_p(S,T)=0$ and
$$\frac{\a_p'(S,T)}{\a_p(S,S)} =  \frac12\sum_{\ell=0}^{a} p^\ell (a 
+b-2\ell+1),$$
where
$$\a_p(S,S) =  \prod_{\ell=1}^{n}(1- (-1)^\ell p^{-\ell}).$$
\end{prop}

In this form, the  
formula in (iv) should hold for any nonsingular fundamental matrix, 
that is,  the relation between derivatives of representation densities and 
intersection multiplicities should continue to hold even in the case of 
improper intersections. More precisely:

\begin{conj}\label{conjecture}
Let  $\bold x=[x_1,\ldots , x_n]\in \VV^n$ be such that $\ZZ_{i,j} 
(\bold x)\neq \emptyset$ and such that the fundamental matrix $T=T 
(\xx)$ is nonsingular. Let $\tilde T=p^{2i-j}T$. Then 
$\ZZ_{i,j}(\bold x)$ is connected and 
$$\chi(\Cal O_{\ZZ_{i,j}(x_1)}\tt^{\mathbb L}\ldots \tt^{\mathbb L} 
\Cal O_{\ZZ_{i,j}(x_n)})=\frac{\alpha'_p(S, \tilde T)}{\alpha_p(S, S)} 
\ .$$
\end{conj}

	The Euler-Poincar\'e characteristic appearing here is indeed finite,  
since it can be shown that $\Cal O_{\ZZ_{i,j}(\bold x)}$ is  
annihilated by a power of $p$. In the case that $\ZZ_{i,j}(\bold x)_ 
{\rm red}$ is of dimension zero, it can be shown that there are no  
higher Tor-terms on the LHS of the above identity \cite{terstiege2},  
so that indeed the statement (iv) of the above theorem confirms the  
conjecture in this case.  Note that the analogue of Conjecture~\ref{conjecture}
in the case of improper intersections of cycles on the Drinfeld space 
was proved in \cite{KRinventiones}. The case of improper intersections of arithmetic Hirzebruch-Zagier cycles is considered by U.~Terstiege in \cite{terstiege} and \cite{terstiege2}.\footnote{Since this paper was submitted, Terstiege has proved the above conjecture in the case $n=3$, cf.~\cite{terstiege3}.}

The layout of the paper is as follows. In section 2, we define the  
moduli space $\Cal N$ and recall the results of Vollaard and Wedhorn  
concerning its structure. In section 3, we introduce special cycles $ 
\ZZ_{i,j}(\bold x)$ on $\Cal N$. In section 4, we prove the statements  
on $\ZZ_{i,j}(\bold x)_{\rm red}$ in our main theorem. The rest of  
the paper is concerned with the determination of  the length of $\Cal  
O_{\ZZ_{i,j}(\bold x),  \xi}$ in the situation of (iv) of the main  
theorem. In section 5, this problem is reduced to a deformation  
problem on $2$-dimensional formal $p$-divisible groups. In section 6,  
we introduce the analogue in our context of Gross's quasi-canonical  
divisors. In section 7, we solve a lifting problem of homomorphisms  
analogous to the one solved by Gross in the classical case. In  
section 8, we use the results of the previous two sections to solve  
the deformation problem of section 5. Finally, in section 9, we relate  
the RHS in (iv) of the main theorem to representation densities.

There are two important ingredients of algebraic-geometric nature  
that are used in our proofs. The first is the results of I. Vollaard  
and T. Wedhorn on the structure of $\Cal N_{\rm red}$, cf. Theorem  
\ref{alggeo}.  The second is the determination,  due to Th. Zink, of  
the length of a certain specific deformation space in equal  
characteristic, cf. Proposition \ref{length}. Both are based on
 Zink's theory of  
displays and their windows, \cite{displays, windows}.

We thank I. Vollaard, T. Wedhorn, and Th. Zink for their  
contributions and for helpful discussions concerning them. We also thank B.~Howard, U.~Terstiege and the referee for their remarks
on the text.
We also  
gratefully acknowledge the hospitality of the Erwin-Schr\"odinger  
Institute in Vienna in 2007, where part of this work was done.

\section{Structure of the moduli space $\Cal N$}

In this section we recall some facts about the moduli space $\Cal N$  
from \cite{vollaard}.
We write $\kay=\Q_{p^2}$ for the unramified quadratic extension of $ 
\Q_p$
and $\OK=\Z_{p^2}$ for its ring of integers. We also write $\FF = \bar 
\F_p$ and $W=W(\FF)$ for its ring of Witt vectors.
There are two embeddings $\ph_0$ and $\ph_1=\ph_0\circ \s$ of $\kay$  
into $W_\Q=W\tt_\Z\Q$.

\subsection{Definition of the moduli space}\label{DefofMod}
Let $(\X,\iota)$ be a fixed supersingular $p$-divisible group of  
dimension $n$ and height $2n$ over $\FF$
with an action $\iota:\OK \rightarrow \End(\X)$ satisfying the {\it  
signature condition
$(r, n-r)$},
\begin{equation}\label{detcond}
\text{\rm char}(\iota(\a),\Lie\, \X)(T) = (T-\ph_0(\a))^{r}(T-\ph_1 
(\a))^{n-r}\in \FF[T].
\end{equation}
Let $\l_{\X}$ be a $p$-principal polarization of $\X$ for which the  
Rosati involution $*$
satisfies $\iota(\a)^*=\iota(\a^\s)$.  The data $(\X,\iota,\l_\X)$ is  
unique up to isogeny.
%\marginpar{Check this.}

Let ${\rm Nilp}_W$ be the category of $W$-schemes on which $p$ is  
locally nilpotent.
The functor $\Cal N =\Cal N(r,n-r)$  associates to a scheme 
$S \in  
{\rm Nilp}_W$ the set of isomorphism
classes of data $(X,\iota,\l_X, \rho_X)$. Here $X$ is a $p$-divisible  
group over $S$ with $\OK$-action $\iota$
satisfying the signature condition $(r, n-r)$,  and  $\l_X$ is a 
$p$-principal polarization of $X/S$, 
  such that the Rosati involution
determined by $\l_X$ induces the Galois involution $\s$ on $\OK$.   
Finally,
$$\rho_X: X \times_{S}\bar S \rightarrow \X\times_{\F}\bar S$$
is an $\OK$-linear quasi-isogeny such that $\rho_X^\vee\circ\l_\X\circ \rho_X$ is a
$\Q_p^\times$-multiple of $\l_X$ in \hfb
$\Hom_{\OK}(X,X^\vee)\tt_\Z\Q$. Here, two such data $(X,\iota,\l_X, \rho_X)$ and
$(X',\iota',\l_{X'}, \rho_{X'})$ are said to be isomorphic if there is a $\OK$-linear isomorphism 
$\alpha: X\to X'$ with $(\alpha\times_W\F)\circ \rho_X=\rho_{X'}$, such that 
$\alpha^\vee\circ \lambda_{X'}\circ \alpha$ is a $\Z_p^\times$-multiple of $\lambda_X$.
This functor is represented by a separated formal scheme $\Cal N$,  
locally formally of finite type over $W$,
\cite{RZ}. Furthermore, $\Cal N$ is formally smooth of dimension $r(n- 
r)$ over $W$, i.e.,   all completed local rings of points in $\Cal N 
(\FF)$ are isomorphic to $W[[t_1,\ldots ,t_{r(n-r)}]]$. Moreover, there is a disjoint sum decomposition
$$\Cal N = \coprod_{i\in \Z}\Cal N_i\ ,$$
where $\Cal N_i$ is the formal subscheme where the quasi-isogeny $ 
\rho_X$ has height $ni$. 

We next review the structure of the underlying reduced subscheme $ 
\Cal N_{{\rm red}}$  of $\Cal N$, following \cite{vollaard}. Note  
that $\Cal N_{{\rm red}}(\FF)=\Cal N(\FF)$.
Let $\MM$ be the (covariant) Dieudonn\'e module of $\X$ and let $N=  
\MM\tt_\Z\Q$ be the associated isocrystal.
Then $N$ has an action of $\kay$ and a skew-symmetric $W_\Q$-bilinear  
form
$\gs{}{}$ satisfying
\begin{equation}\label{skew}
\gs{Fx}{y} = \gs{x}{Vy}^\s,
\end{equation}
and with
$\gs{\a x}{y}=\gs{x}{\a^\s y}$,
for $\a\in \kay$.

\begin{prop} {\rm (\cite{vollaard}, Lemmas 1.4, 1.6)}
  There is a bijection between $\Cal N_i(\FF)$
and the set of  $W$-lattices $M$ in $N$ such that $M$ is  stable  
under $F$, $V$,
and $\OK$, and with the following properties:
\begin{equation}\label{detcond2}
\text{\rm char}(\a,M/VM)(T) = (T-\ph_0(\a))^{r}(T-\ph_1(\a))^{n-r}\in  
\F[T],
\end{equation}
and
\begin{equation}
M = p^i M^\perp\ ,
\end{equation}
where
\begin{equation}
M^\perp = \{\, x\in N\mid \gs{x}{M} \subset W\, \}. 
\end{equation}
\end{prop}

The isomorphism $\OK\tt_{\Z_p} W \isoarrow W\oplus W$ yields a  
decomposition
$M = M_0\oplus M_1$  into rank $n$ submodules, and $F$ and $V$ have  
degree $1$ with respect to this decomposition.
Also $M_0$ and $M_1$ are isotropic with respect to $\gs{}{}$.
Moreover, the determinant condition (\ref{detcond2}) is equivalent to  
the chain condition
\begin{equation}\label{chain}
p M_0 \subover{n-r} FM_1 \subover{r} M_0,
\end{equation}
or, equivalently,
$$p M_1 \subover{r} FM_0 \subover{n-r} M_1.$$
(Here the numbers above the inclusion signs indicate the lengths of  
the respective cokernels.) Note that $M_0= M\cap N_0$ and $M_1= M\cap  
N_1$ for the analogous decomposition
$N=N_0\oplus N_1$.

Since the isocrystal $N$ is supersingular, the operator $\tau= V^{-1} 
F = pV^{-2}$ is a $\s^2$-linear
automorphism of degree $0$ and has all slopes $0$.
Let
$$C = N_0^{\tau=1}$$
be the space of $\tau$-invariants, so that $C$ is a $n$-dimensional $ 
\Q_{p^2}$-vector space and
$$N_0= C\tt_{\Q_{p^2},\ph_0}W_\Q.$$
For $x$ and $y\in N$, let
$$\hgs{x}{y} = \delta^{-1}\gs{x}{Fy},$$
where
$\delta\in \Z_{p^2}^\times$ with $\delta^\s=-\delta$ is fixed once  
and for all. Note that the corresponding form in \cite{vollaard}
is taken without the scaling by $\delta$.
By (\ref{skew}),
$$\hgs{x}{y} = \hgs{y}{\tau^{-1}(x)}^\s,$$
and hence
$$\hgs{\tau(x)}{\tau(y)} = \hgs{x}{y}^{\s^2}.$$
Thus, $\hgs{\,}{}$ defines a $\Q_{p^2}$-valued hermitian form on $N^ 
{\tau=1}$
and, in particular, on $C$.  Note that, since the polarization form $ 
\gs{}{}$ is non-degenerate on
$N$ with $N_0$ and $N_1$ isotropic subspaces, and since $FN_0=N_1$,  
the hermitian form
$\hgs{\,}{}$ is non-degenerate on $C$.

For a $W$-lattice $L\subset N_0$, let
\begin{equation}
L^\vee=\{\,x\in N_0\mid \hgs{x}{L}\subset W\,\} = (FL)^\perp.
\end{equation}
Then $(L^\vee)^\vee=\tau(L)$.
\begin{lem}\label{perpslemma}
For an $\OK$-stable  Dieudonn\'e module $M = M_0\oplus M_1$,
$$M= p^i M^\perp \quad \iff \quad FM_1 = p^{i+1}M_0^\vee.$$
\end{lem}
\begin{proof}
The condition $M= p^i M^\perp$ is equivalent to the condition
\begin{equation}\label{dual1}
\gs{M_0}{M_1}=p^i W.
\end{equation}
Since $VF=FV=p$, this is, in turn, equivalent to
$\gs{FM_1}{FM_0} = p^{i+1}W$,
by (\ref{skew}).  This last identity can be rewritten as
\begin{equation}
\hgs{FM_1}{M_0}= p^{i+1}W,
\end{equation}
as claimed.
\end{proof}

Note that the lattice $M_0$ in $N_0$ determines the lattice $M_1$
in $N_1$, either by (\ref{dual1}) or by the condition $FM_1 = p^{i+1} 
M_0^\vee$.

\begin{prop}\label{inken1}{\rm (\cite{vollaard}, Prop. 1.11, Prop.  
2.6 a))} There is a bijection between $\Cal N_i(\FF)$ and the set
of $W$-lattices
$$
\Cal D_i= \Cal D_i(C) = \big\{\ A\subset N_0 \ \big\vert\
p^{i+1}A^\vee \subover{r} A\subover{n-r} p^i A^\vee\ \big\},
$$
given by mapping the Dieudonn\'e lattice $M= M_0\oplus M_1$
to the $W$-lattice $A=M_0$ in $N_0$.  Under this correspondence
$$FM_1 = p^{i+1}A^\vee.$$
\end{prop}

\begin{rem} As observed in \cite{vollaard}, Lemma~1.16, for an $\OK$- 
lattice $A$ in $C$ with
$$pA^\vee \subover{r} A \subover{n-r} A^\vee,$$
there is a $\OK$-basis for $A$ for which the hermitian form $\hgs{}{} 
$ has matrix
$\diag(1_r, p 1_{n-r})$.
\end{rem}

\begin{rem}\label{rem10case} We record the two simplest examples for  
later use.  We first consider the case of signature $(1,0)$. Let us call $\YY$ the base point used to define $\mathcal N$ above. In this case
we have $C = \Q_{p^2}\cdot 1_0$ with hermitian form given by
$\hgs{1_0}{1_0} =  1$, and
$N_0 = W_\Q\cdot1_0$.
Moreover, $\Cal N_i(\F) \simeq \Cal D_i$ is empty for $i$ odd, and
$\Cal N_{2i}(\F) \simeq \Cal D_{2i}$ consists of a single point
corresponding to the $W$-lattice $A=p^i W\cdot 1_0$ in $N_0$, where  
$A^\vee = p^{-i}W\cdot 1_0$.
Let $\MM^0= \MM^0_0+\MM^0_1$ be the Dieudonn\'e  
module of $\YY$. Then,
$\MM^0_0=W\cdot1_0$,  $\MM^0_1=W\cdot1_1$,  $F 1_1 = p1_0$,  $F  
1_0=1_1$, and the polarization $\gs{\ }{\ }^0$ is given by
\begin{equation}
\gs{1_0}{1_1}^0 = \delta.
\end{equation}
Note that, on $N_0=\MM^0\tt_\Z\Q$,  $\hgs{1_0}{1_0} = \delta^{-1}\gs 
{1_0}{F1_0}^0 =1$, as required. Also note that $\End(\YY)$ can be  
identified with the ring of integers
$O_D$ in the quaternion division algebra over $\Q_p$, and that the  
endomorphism
\begin{equation}
\Pi :\ 1_0\mapsto p1_1, \ 1_1\mapsto 1_0
\end{equation}
is a uniformizer of $O_D$. \hfb
Now $\YY$ is a formal $\OK$-module over $\FF$. Hence there is a  
unique lift $Y$  of $\YY$ to $W$ as a formal $\OK$-module. In  
particular,  $\Cal N_{2i}=\Spf\ W$ for every $i$.

In the case of signature $(0,1)$,
again $C = \Q_{p^2}\cdot \bar1_0$ but now with hermitian form given by
$\hgs{\bar1_0}{\bar1_0} =  p$.  Moreover, $\Cal N_i(\F) \simeq \Cal  
D_i$ is empty for $i$ odd, and
$\Cal N_{2i}(\F) \simeq \Cal D_{2i}$ consists of a single point
corresponding to the $W$-lattice $A=p^i W\bar 1_0$ in $N_0$, where $A^ 
\vee = p^{-i-1}W\bar 1_0$.
We write  $\YYbar$ for the $p$-divisible group over $\FF$  
corresponding to the unique point of
$\Cal N_0(\F)$, and let $\MMbar^0= \MMbar^0_0+\MMbar^0_1$ be its  
Dieudonn\'e module. Then,
$\MMbar^0_0=W\cdot\bar1_0$,  $\MMbar^0_1=W\cdot\bar1_1$,  $F \bar1_1  
=  \bar1_0$,  $F \bar1_0=p\bar1_1$, and the polarization is given by
\begin{equation}
\gs{\bar1_0}{\bar1_1}^0 = \delta.
\end{equation}
On $\bar N_0=\MMbar^0\tt_\Z\Q$,  we have $\hgs{\bar1_0} 
{\bar1_0} = \delta^{-1}\gs{\bar1_0}{F\bar1_0}^0 =p$. 

Note that the formal $\OK$-module $\YYbar$ is obtained from $\YY$ by  
identifying the underlying $p$-divisible groups, but changing the $\OK 
$-action by precomposing it with the Galois automorphism $\sigma$ of $ 
\OK$. On the level of Dieudonn\'e modules,
this identification is given by switching the roles of $\MM^0_0$ and $ 
\MM^0_1$. In particular, the canonical lift $\bar Y$ of $\YYbar$ to $W$ is  
isomorphic to the canonical lift $Y$, with the conjugate $\OK$-action.

\end{rem}

\subsection{The case of signature $(1,n-1)$}
 From now on, we assume that $r=1$ so that our signature is $(1,n-1)$.

The structure of $\Cal D_i$ is best described in terms of strata  
associated to vertices
of the building for the unitary group $\text{\rm U}(C)$ of the  
hermitian space $(C$, $\hgs{}{})$.
Recall \cite{vollaard} that a vertex of level
  $i$ is a $\tau$-invariant $W$-lattice $\L$ in $N_0$
with
$$p^{i+1}\L^\vee\subover{t} \L\subset p^i\L^\vee.$$
Here we are identifying lattices in $C$ with $\tau$-invariant lattices
in $N_0$. 
The type of a vertex $\L$ is the index $t=2d+1$ of $p^{i+1}\L^\vee$  
in $\L$ (which is always an odd integer between $1$ and $n$).

\begin{lem}\label{zink} {\rm (quantitive version of Zink's Lemma, cf.  
\cite{vollaard}, Lemma 2.2.)}
For $A\in \Cal D_i$, there exists a minimal $d$, $0\le d\le \frac{n-1} 
2$, such that
the lattice
$$\L =\L(A):= A + \tau(A)+\dots+\tau^d(A)$$
is $\tau$-invariant. Then $\L$ is a vertex of level $i$, i.e., 
$$p^{i+1}\L^\vee \subseteq \L\subseteq p^i\L^\vee\ . $$
Moreover
$$p^{i+1}\L^\vee \subseteq p^{i+1}A^\vee \subover{1} A\subseteq \L 
\subseteq p^i\L^\vee,$$
and the index of $p^{i+1}\L^\vee$ in $\L$, the type of $\L$, is $2d+1$.
\end{lem}

A lattice $A\in \Cal D_i$ is {\it superspecial} if $\tau(A)=A$, so that $A= 
\L$ is also
a vertex of type $1$. In general, $\L(A)$ is the smallest $\tau$-invariant lattice  
containing $A$ and, equivalently,
$p^{i+1}\L^\vee$ is the largest $\tau$-invariant lattice contained in  
$p^{i+1} A^\vee$.

For a vertex $\L$ of level $i$, let
$$\Cal V(\L) = \{\ A\in \Cal D_i\ \mid \ A\subset \L\ \}.$$
Equivalently, $p^{i+1}\L^\vee\subset p^{i+1}A^\vee$.

As $\L$ ranges over
  the vertices of level $i$, the subsets $\Cal V(\L)$ of  $\Cal D_i$ are organized in a coherent way, as  
follows.

  For two vertices $\L$ and $\L'$ of level $i$,
  \begin{equation}\label{}
  \Cal V(\L)\subset \Cal V(\L') \iff \L\subset \L',
\end{equation}
cf. \cite{vollaard}, Prop. 2.7. In this case we have 
$t(\L)\leq t(\L')$ for the types.

Let
$$\Cal V(\L)^o= \Cal V(\L)\setminus \bigcup_{\L'\subsetneqq \L} \Cal V 
(\L').$$
Then by \cite{vollaard}, Prop. 3.5
\begin{equation}
\Cal V(\L)= \bigcup_{\L'\subset \L} \Cal V(\L')^o,
\end{equation}
where the (finite) union on the right hand side is disjoint and runs  
over all vertex lattices of level
$i$  contained in $\L$ (these are then of type $t(\L')\leq t(\L)$).

The subsets $\Cal V(\L)$ and $\Cal V(\L)^o$  have an algebraic-geometric
 meaning.
\begin{theo}\label{alggeo}{\rm (Vollaard, Wedhorn)}
Let $\Cal N_i$ be non-empty.  \hfb
\noindent (i) $\Cal N_{i,{\rm red}}$ is connected.\hfb
\noindent (ii) For any vertex $\L$ of level $i$,  $\Cal V(\L)$ is the set of $\FF$-points of a closed  
irreducible subvariety of dimension
$\frac{1}{2}(t(\L)-1)$ of $\Cal N_{i,{\rm red}}$, and the inclusions $ 
\Cal V(\L')\subset \Cal V(\L)$ for $\L'\subset \L$
are induced by closed embeddings of algebraic varieties over $\FF$.
\end{theo}
Vollaard \cite{vollaard} has proved this for signature $(1, s)$ with  
$s=1$ or $2$. The general case is in \cite{vollwedhorn}. In \cite 
{vollwedhorn} it is also proved that the variety
corresponding to $\Cal V(\L)$ is  smooth.\hfb
We note that $\Cal N_i$ is always non-empty if $n$ is even; if $n$ is odd, then $\Cal N_i$
is non-empty if and only if $i$ is even, cf. \cite{vollaard}, Prop. 1.22.  

We conclude this section with the following observation about scaling,
which will be useful later.

\begin{lem}
(i) If $A\in \Cal D_i$, then, for $k\in \Z$, $p^k A\in \Cal D_{i+2k}$.  
\hfb
(ii) If $\L$ is a vertex of level $i$, then $p^k\L$ is a vertex of  
level $i+2k$. \hfb
(iii) If $\L$ is a vertex of level $i$ for the hermitian space $(C$, $ 
\hgs{}{})$, then $\L$ is a vertex of
level $i+k$ for the hermitian space $(C$, $p^k\hgs{}{})$.
\end{lem}
\begin{proof}
  Here, in cases (ii) and (iii), note that
$$p^{i+1}\L^\vee\subset \L\subset p^i\L^\vee \quad\iff\quad p^{i+1}W  
\subset \hgs{\L}{\L}\subset p^iW.$$
\end{proof}

\section{Cycles in the moduli space $\Cal N$.}

Let $(\YYbar,\iota)$ be the $p$-divisible group of dimension $1$ and  
height $2$ over $\FF$,
with an action $\iota:\OK\rightarrow \End(\YYbar)$ of $\OK$ and with  
principal polarization
$\l_{\YYbar}$ satisfying the
signature condition  $(0, 1)$, cf. Example \ref{rem10case}.
Let $\Cal N^0 =\Cal N(0,1)$ be the corresponding moduli space as in  
section 1.
Recall that, by Remark~\ref{rem10case},
$$\Cal N^0 = \coprod_{i\in \Z} \Cal N^0_{2i},$$
and that  $\Cal N^0_{2i}=\Spf\ W$.  For example, the unique point of $ 
\Cal N_0^0(\F)$
corresponds to $(\YYbar,\iota, \l_{\YYbar}, \rho_{\YYbar})$ where $ 
\rho_{\YYbar}$ is the identity map.
\begin{defn}
The {\it space of  special homomorphisms} is the $\kay$-vector space
$$\VV:=\Hom_{\OK}(\YYbar,\X)\tt_\Z\Q.$$
\end{defn}

For  $x, y\in \VV$,
we let
\begin{equation}\label{hermVV}
h(x,y) = \l_{\YYbar}^{-1}\circ \hat{y}\circ \l_{\X}\circ x \in \End_ 
{\OK}(\YYbar)\tt\Q \overset{\iota^{-1}}{\isoarrow} \kay.
\end{equation}
This hermitian form  is $\OK$-valued on
the lattice $\LL:=\Hom_{\OK}(\YYbar,\X)$.

\begin{defn}
(i) For a given special homomorphism $x\in \VV$, define the {\it  
special cycle
$\ZZ(x)$ associated to $x$ in} $\Cal N^0\times \Cal N$
as the subfunctor of collections  $\xi=(\bar Y,\iota,\l_{\bar Y}, \rho_ 
{\bar Y}, X,\iota,\l_X,\rho_X)$ in
  $(\Cal N^0\times \Cal N)(S)$
such that the quasi-homomorphism
$$\rho_X^{-1}\circ x\circ \rho_{\bar Y}: \ \bar Y \times_S \bar S   
\lra X\times_S \bar S$$
extends to a homomorphism from $\bar Y$ to $X$.  Here $\bar S = S  
\times_W \FF$ is the special
fiber of $S$. \hfb
(ii) More generally, for a fixed $m$-tuple $\xx = [x_1,\dots. x_m]$  
of special homomorphisms $x_i\in \VV$,
the {\it associated special cycle} $\ZZ(\xx)$ is the subfunctor of $\Cal  
N^0\times \Cal N$  of collections  $\xi=(\bar Y,\iota,\l_{\bar Y}, \rho_ 
{\bar Y}, X,\iota,\l_X,\rho_X)$ in
  $(\Cal N^0\times \Cal N)(S)$
such that the quasi-homomorphism
$$\rho_X^{-1}\circ \xx\circ \rho_{\bar Y}: \ \bar Y^m \times_S \bar S   
\lra X\times_S \bar S$$
extends to a homomorphism from $\bar Y^m$ to $X$.\hfb
(iii) For $i$ and $j\in \Z$, let $\ZZ_{ij}(\xx)$ be the formal  
subscheme of $\ZZ(\xx)$ whose projection to
$\Cal N^0$ (resp. $\Cal N$) lies in $\Cal N^0_{2i}$ (resp. $\Cal N_j 
$), i.e.,
$$\begin{matrix} \ZZ_{ij}(\xx) &\lra& \ZZ(\xx)\\
\nass
\downarrow&{}&\downarrow\\
\nass
\Cal N^0_{2i}\times \Cal N_j&\lra&\Cal N^0\times \Cal N.\end{matrix}
$$
\end{defn}

We note that $\Cal N^0_{2i}$ has been identified with $\Spf \ W$, via  
the canonical lift
$\bar Y$ of $\YYbar$, with its $\OK$-action. Hence $\ZZ_{ij}(\xx)$  
can be identified with
a closed formal subscheme of $\Cal N_j$.

\begin{rem} (i) It is clear from the definition that $\ZZ(\xx)$  
depends only on the orbit of the vector $\xx$
under the right action of $\GL_m(\OK)$, which acts as automorphisms  
of $\YYbar^m$.

\noindent (ii) The definition of the special cycles is compatible  
with intersections. Specifically,
the intersection of $\ZZ(\xx)$ and $\ZZ(\yy)$ is the locus where the  
whole collection $[\xx,\yy] =[x_1,\dots, x_m, y_1,\dots, y_{m'}]$
deforms, i.e.,
$$\ZZ(\xx)\cap \ZZ(\yy) = \ZZ([\xx,\yy]).$$
\end{rem}
\begin{rem}\label{scalingfor}
We note that  $\Cal N^0_{2i}$ has been identified with  $\Cal N^0_{0}$. Explicitly,
$(\bar Y, \iota, \lambda_{\bar Y}, \rho_{\bar Y})$ in  $\Cal N^0_{2i}$ is sent to 
$(\bar Y, \iota, \lambda_{\bar Y}, p^{-i}\rho_{\bar Y})$. Under this identification 
the  subfunctor $ \ZZ_{ij}(\xx)$ of $\Cal N^0_{2i}\times \Cal N_j$ is identified with the  subfunctor $ \ZZ_{0,j-2i}(\xx)$ of 
$\Cal N^0_{0}\times \Cal N_{j-2i}$.  Here the point is that the compositions $\rho_X^{-1}\circ \xx\circ \rho_{\bar Y}$ and $(p^{-i}\rho_X)^{-1}\circ \xx\circ (p^{-i}\rho_{\bar Y})$
coincide.\hfb
For the same reason, $ \ZZ_{i, j}(\xx)$ can be identified with $ \ZZ_{0, j}(p^i\xx)$. 
\end{rem}

\begin{prop}\label{divisor}
The functor $\ZZ(\xx)$ is represented by a closed formal subscheme of  
$\Cal N^0\times \Cal N$. In fact,  $\ZZ(x)$ is a relative divisor in $ 
\Cal N^0\times \Cal N$ (or empty) for any $x\in \VV \setminus \{0\}$.

{\rm Recall that a relative divisor is a closed formal subscheme, locally defined  
by one equation, which is neither a unit nor divisible by $p$.}
\end{prop}
\begin{proof}
The first statement follows from \cite{RZ}, Proposition 2.9. To prove the  
second statement, it suffices to prove that $\ZZ_{i, j}$ is a relative  
divisor in $\Cal N_j$ (via the second projection in $\Cal N^0\times  
\Cal N$). By following the proof of the corresponding statement in  
\cite{terstiege}, Prop. 4.5, we see that, in order to prove that $\ZZ_{i, j}(x)$ is  
locally defined by the vanishing of one equation, it suffices to  
prove the following statement. Let $A$ be a $W$-algebra and let  
$A_0=A/I$, where the ideal $I$ satisfies $I^2=0$. We equip $I$ with  
trivial divided powers. We assume given a morphism $\phi: \Spec A\to  
\Cal N_j$ whose restriction to $\Spec A_0$ factors through $\ZZ_{i, j} 
$. Then the obstruction to factoring the given morphism through $\ZZ_ 
{i, j}$ is given by the vanishing of one element in $I$.

The value $\Cal D_{\bar Y}$ of the Dieudonn\'e crystal  of  $\bar Y 
\times_{\Spec W}\Spec A$ on $(A, A_0)$ is given by $\MMbar\otimes_WA 
$, and is equipped with its Hodge filtration
\begin{equation}\label{hodgeY}
0\to \Cal F_{\bar Y}\to \Cal D_{\bar Y}\to {\rm Lie}\  \bar Y\otimes_WA\to 0 ,
\end{equation}
where $\Cal F_{\bar Y}$ is generated by the element $\bar 1_0$.  
Similarly, the Dieudonn\'e
crystal $\Cal D_X$ of the pullback to $A$ of the universal object $ 
(X, \iota, \lambda)$ over $\Cal N$  comes with its Hodge filtration
\begin{equation}\label{hodgeX}
0\to \Cal F_X\to \Cal D_{X}\to {\rm Lie}\ X\to 0 .
\end{equation}
The fact that the restriction of $\phi$ to $\Spec A_0$ factors through $\ZZ_{i, j}(x)$ implies that there  
is an $\OK$-linear homomorphism $\alpha: \Cal D_{\bar Y}\to \Cal D_X$  
of $A$-modules, which respects the filtrations (\ref{hodgeY}) and  
(\ref{hodgeX}) after tensoring with $A_0$. We need to show that the  
condition that $\alpha$ respect the filtrations (\ref{hodgeY}) and 
(\ref{hodgeX}) is locally defined by one equation. However, this  
condition is obviously that $\alpha(\bar 1_0)\in \Cal F_X$,
i.e., that the image of $\alpha(\bar 1_0)$ in the degree zero component
$({\rm Lie}\ X)_0$ vanishes.
However, thanks to the signature condition, $({\rm Lie}\ X)_0$ is a  
locally free $A$-module of rank $1$. After choosing a local  
generator of $({\rm Lie}\ X)_0$, we may identify $\alpha(\bar 1_0)$ with an element in $A$
with zero image in $A_0$. Hence the condition is described by the  
vanishing of one element in the ideal $I$.

We still have to show that this element is non-trivial, and  that  
it is not divisible by $p$.
We first note the following simple fact.
\begin{lem}
Let $\frak X$ be a formal scheme such that $\Cal O_{\frak X, x}$ is  
factorial for each
$x\in \frak X_{\red}$. Let $g\in\Gamma(\frak X, \Cal O_{\frak X})$  
with $g|\frak X_{\red}\equiv 0$
and such that $g\in\Cal O_{\frak X, x}$ is an irreducible element for  
each
$x\in\frak X_{\red}$. Let $f\in\Gamma(\frak X, \Cal O_{\frak X})$ and  
consider the subset
\begin{equation*}
V = \{x\in\mathfrak X_{\rm red}\mid  g \text{{ \rm divides }} f  \text 
{{ \rm in } } \Cal O_{\mathfrak X, x}\}\ .
\end{equation*}
Then $V$ is open and closed in $\frak X_{\rm red}.$
\end{lem}
\begin{proof}
Consider the ideal sheaf
\begin{equation*}
\frak a = \{h\in\Cal O_\frak X\mid hf\in g\Cal O_\frak X\}\ .
\end{equation*}
Then $x\in V\iff{\frak a}_x = \Cal O_{\frak X, x}\iff x \notin {\rm  
Supp} (\Cal O_\frak X /\frak a)$.
Hence $V$ is open. To show that $V$ is closed, let $x\in V$ and let $x'$
be a specialization of $x$. In $\Cal O_{\frak X, x}$ we have an identity
\begin{equation*}
\frac{f}{g} = \frac{h}{s}\quad ,\quad s\in\Cal O_{\frak X, x'} 
\setminus \frak p_x\ .
\end{equation*}
Hence we obtain an identity in $\Cal O_{\frak X, x'}$,
\begin{equation*}
f s = g h\ .
\end{equation*}
If $g\mid f$ in $\Cal O_{\frak X, x'}$, then $x'\in V$. Otherwise,  
since $g$ is
irreducible in $\Cal O_{\frak X, x'}$, we have $g\mid s$. But then
$s\in g\Cal O_{\frak X, x'}\subset\frak p_x$, a contradiction.
\end{proof}
Now we prove that a local equation for $\ZZ_{i, j}(x)$ is not divisible  
by $p$.
Otherwise, by the previous lemma, and since, by Theorem \ref 
{alggeo} (i), $\Cal N_j$ is connected, it would
follow that $\Cal N_j(\FF)= \ZZ_{i, j}(x)(\FF)$. We now distinguish the
cases $n\geq 3$ and $n\leq 2$.

In the case $n\geq 3$ we appeal to Proposition \ref{proj} below (of  
course, the
proof of this proposition does not use the statement we are in the  
process
of proving). According to this proposition, the inclusion
$\Cal N_j (\FF)\subset \ZZ_{i, j}(x) (\FF)$ implies
\begin{equation}\label{intersection}
x\in\bigcap_\Lambda\ p^{j+1}\Lambda^\vee\ ,
\end{equation}
where $\Lambda$ runs through all vertices of level $j$. By the next  
lemma,
the intersection (\ref{intersection}) is trivial.
\begin{lem}
Let $n\geq 3$. Then
\begin{equation*}
\bigcap_\Lambda \ \Lambda\ = (0)\ .
\end{equation*}
Here the intersection runs over all vertices of level $j$.
\end{lem}
\begin{proof}
 By the reduction argument right after Corollary
\ref{integralityT} below, we may assume $j = 0$. First let $n$ be odd. By the parity condition  of \cite{vollaard} recalled in 
(\ref{Codd}) below, we may
choose a basis of $C$ as follows. Let $n-1 = 2k$ with $k\geq 1$. We  
choose
a basis $e, f_{\pm 1}, \ldots , f_{\pm k}$ with
\begin{equation*}
\{e, e\} = 1,\ \{e, f_{\pm i}\} = 0\ ,\quad \{f_i, f_j\} = p\delta_ 
{i, -j}\ .
\end{equation*}
Let ${\bf a} = (a_1, \ldots , a_k)\in\Z^k$. Set
\begin{equation*}
\Lambda =\Lambda_{\bf a}= [e, p^{a_1} f_1, \ldots , p^{a_k} f_k, p^{- 
a_k} f_{-k}, \ldots , p^{-a_1} f_{-1}]\ .
\end{equation*}
Then
\begin{equation*}
\Lambda^\vee = [e, p^{a_1-1} f_1, \ldots , p^{a_k-1} f_k, p^{-a_k-1}  
f_{-k}, \ldots , p^{-a_1-1} f_{-1}]\ .
\end{equation*}
Hence $\Lambda$ is a vertex of level $0$. By varying
${\bf a}\in\Z^k$, the intersection of these vertices is zero, which  
implies
the assertion in this case.

  Now let $n$ be even. Let $n-1 = 2k+1$, with $k\geq 1$.
We choose a basis $e, f_{\pm 1}, \ldots , f_{\pm k}, g$ with
\begin{equation*}
\{e, e\} = 1, \{g, g\} = p, \{e, g\} = \{e, f_{\pm i}\} = \{g, f_{\pm  
i}\} = 0, \{f_{\pm i}, f_{\pm j}\} = p\delta_{i, -j}\ .
\end{equation*}
Then for ${\bf a} = (a_1, \ldots , a_k)\in\Z^k$, the lattice
\begin{equation*}
\Lambda_{\bf a} = [e, p^{a_1} f_1, \ldots , p^{a_k} f_k, g, p^{-a_k}  
f_{-k}, \ldots , p^{-a_1} f_{-1}]
\end{equation*}
is a vertex of level $0$, and the assertion follows as before.
\end{proof}

This proves the assertion for $n\geq 3$. For $n = 1$, the assertion  
is trivial.
Now let $n = 2$. In this case $\Cal N_{\rm red}$ is a discrete set of points, 
but the local rings are two-dimensional, hence the assertion is non-trivial in this case.
  Let $y\in\ZZ_{i,j} (x)(\FF)$.
By the uniformization theorem \cite{RZ}, Thm. 6.30, the complete  
local ring $\widehat{\Cal O}_{\Cal N_j, y}$ is isomorphic to the  
completion
of the local ring of a closed point of the integral model of the  
Shimura variety
attached to ${\rm GU}(1, 1)$. Hence by Wedhorn's theorem 
\cite{mu-ord}, the pullback of the
universal $p$-divisible group $X$ to the generic point of
$\Spec (\widehat{\Cal O}_{\Cal N_j, y}\otimes\FF)$ is ordinary, i.e.,  
isogenous to
$\widehat{\G}_m^2\times (\Q_p/\Z_p)^2$. Hence there is no non-trivial
homomorphism from $\bar{\YY}$ into this $p$-divisible group, hence  
the closed
subscheme of $\Spec\widehat{\Cal O}_{\Cal N_j, y}$ cut out by an  
equation of
$\ZZ_{i, j} (x)$ does not contain the special fiber, as was to be shown.
\end{proof}

To study the set $\ZZ(\xx)(\FF)$ of $\FF$-points of a special cycle $ 
\ZZ(\xx)$, we
apply the construction of the previous section and reformulate things  
in terms of the hermitian space $C$.

We begin by describing the space $\VV$ of special homomorphisms.
Recall from Remark \ref{rem10case} the Dieudonn\'e module $\MMbar^0=  
\MMbar^0_0+\MMbar^0_1=W\bar 1_0+W\bar 1_1$  of $\YYbar$.
\begin{comment}
$$M^0 = W^2, \quad F = \Pi\s, \quad \Pi= \begin{pmatrix}{}&1\\ p&{} 
\end{pmatrix}, \quad
\iota(\a) = \begin{pmatrix} \a&{}\\{}&\bar\a\end{pmatrix},$$
with polarization form given by the determinant,
$$\gs{(a_0,a_1)}{(b_0,b_1)} = \delta(a_0b_1-a_1b_0).$$
%up to scaling.
In particular, $F 1_1 = 1_0$
and $F 1_0=p 1_1$.
\end{comment}
A special homomorphism $x\in \VV$ corresponds to a homomorphism,  
which we also denote by $x$,
from $\bar N^0$ to $N$. This homomorphism has degree zero with  
respect to the grading given by the
$\OK$-action, and so we may write $x=x_0+x_1$ where
$x_0:\bar N^0_0 \rightarrow N_0$ and $x_1:\bar N^0_1\rightarrow N_1 
$.  Moreover, since $x$ is $F$-linear,
$x_1 F= Fx_0$, so that $x_0$ determines $x_1$.  In particular, $x$ is  
determined by $x_0(\bar 1_0)$.
Note that $x_0(\bar 1_0)\in C = N_0^{\tau=1}$, since $x$ commutes with
$F$ and $V$, and $\bar 1_0\in C^0 = (\bar N^0_0)^{\tau=1}$.

The hermitian form on $\VV$  defined by (\ref{hermVV}) can be written as
$$h(x,y) = \iota^{-1}(y^*\circ x),$$
where $y^*$ is the adjoint of $y$ with respect to the polarizations,  
i.e., for $u^0\in \bar N^0$ and $u\in N$,
$$\gs{y(u^0)}{u} = \gs{u^0}{y^*(u)}^0.$$

\begin{lem} For $x = x_0+x_1$ and $y= y_0+y_1$ in $\VV$,
$$y^*\circ x = y_1^*\circ x_0 + y_0^*\circ x_1 \in \OK\tt W\simeq W 
\oplus W,$$
with
$$y_1^*\circ x_0  = p^{-1}\hgs{x_0(\bar 1_0)}{y_0(\bar 1_0)},$$
and
$$y_0^*\circ x_1 = p^{-1}\hgs{y_0(\bar 1_0)}{x_0(\bar 1_0)}.$$
\end{lem}

\begin{proof} Writing $y_1^*\circ x_0(\bar 1_0) = \a\,\bar 1_0$, we have
\begin{align*}
-\a \delta&= \gs{\bar 1_1}{y_1^*\circ x_0(\bar 1_0)}^0\\
\nass
{}&= \gs{y_1(\bar 1_1)}{x_0(\bar 1_0)}\\
\nass
{}&= \gs{Fy_0F^{-1}(\bar 1_1)}{x_0(\bar 1_0)}\\
\nass
{}&=-p^{-1}\delta \hgs{x_0(\bar 1_0)}{y_0(\bar 1_0)}.
\end{align*}
The component $y_0^*\circ x_1$ is found in the same way.
\end{proof}

In summary, we have proved the following.

\begin{lem}\label{relVC}
 There is an isomorphism
$$\VV = \Hom_{\OK}(\YYbar,\X)\tt\Q \isoarrow C, \qquad x\mapsto x_0 
(\bar 1_0).$$
The hermitian forms on the two spaces are related by
$$h(x,y) =  p^{-1}\hgs{x_0(\bar 1_0)}{y_0(\bar 1_0)}.$$\qed
\end{lem}

\begin{prop}\label{proj}
For $\xx = [x_1,\dots, x_m]\in \VV^m$, let
$L$ be the $W$-submodule in $N_0$ spanned by the
components of
$$\xx_0(\bar 1_0) =[(x_1)_0(\bar 1_0), \dots, (x_m)_0(\bar 1_0)]\in  
C^m.$$
Then the image of the projection of $\ZZ_{i,j}(\xx)$ to
$\Cal N_j(\FF)\simeq \Cal D_j$ is
$$\Cal W_{i,j}(\xx):=\{\ A\in \Cal D_j\mid p^iL\subset p^{j+1} A^\vee 
\ \}.$$
\end{prop}
\begin{proof}
Recall that  $\Cal N^0_0(\FF)$ consists of a single point  
corresponding to  $\YYbar$ with Dieudonn\'e
lattice given above, cf. Remark~\ref{rem10case}.  Similarly, $\Cal  
N^0_{2i}(\FF)$ consists of a single point
corresponding to $\bar Y$ with Dieudonn\'e lattice $\bar M^0=p^i  
\MMbar^0$ in $\bar N^0$.
A special homomorphism $x\in \VV$, extends  to $\bar Y \rightarrow X$  
if and only if
$x(p^i \MMbar^0) \subset M$, where $M$ is the Dieudonn\'e lattice of  
$X$.
The latter condition is equivalent to $p^i x_0(\bar 1_0)\in M_0$ and  
$p^i x_1(\bar 1_1)\in M_1$.
Recall that $FM_1 = p^{j+1} A^\vee \subset A$.  Then,
$p^i x_1(\bar 1_1)\in M_1$ if and only if $p^iFx_1(\bar 1_1)\in  
FM_1=p^{j+1} A^\vee$.
But, since $Fx_1= x_0 F$ and $F(\bar 1_1)=\bar 1_0$, this last  
condition is equivalent to
$p^i x_0(\bar 1_0)\in p^{j+1} A^\vee$, which, in turn, implies the  
condition $p^i x_0(\bar 1_0)\in M_0=A$.
Thus, a  collection $\xx$ extends if, for each component $x_r$,  $p^i(x_r)_0(\bar 1_0)$ lies  
in $p^{j+1}A^\vee$, or, equivalently,
$p^iL\subset p^{j+1}A^\vee$.
\end{proof}

We call the hermitian matrix
$$T=h(\xx,\xx) = \big(h(x_i,x_j)\big)\in\Herm_m(\kay)$$
the {\it fundamental matrix} determined by $\xx$. 
 There is a variant of it which will also be useful in the sequel. Namely, when considering $\ZZ_{i,j}(\xx)$, where the fundamental matrix of $\xx$ is $T$, we will call the matrix $\tilde T=p^{2i-j} T$ the {\it scaled fundamental matrix} attached to $\ZZ_{i,j}(\xx)$. 

\begin{cor}\label{integralityT} If $\ZZ_{i,j}(\xx)(\FF)$ is 
non-empty, then the corresponding scaled fundamental matrix $\tilde T$ is  
integral, i.e., 
$$\tilde T  \in \Herm_m(\OK).$$
\end{cor}
\begin{proof} The components of the matrix $h(\xx,\xx)$
have the form $p^{-1}\hgs{x_0(\bar 1_0)}{y_0(\bar 1_0)}$ with
$x_0(\bar 1_0)$ and $y_0(\bar 1_0)$ contained in $p^{j-i+1} A^\vee  
\subset p^{-i}A$. Note that the matrix determined by the
components of $\xx_0(\bar 1_0)$ is then
$$\hgs{\xx_0(\bar 1_0)}{\xx_0(\bar 1_0)} =  \big(\hgs{(x_i)_0(\bar  
1_0)}{(x_j)_0(\bar 1_0)}\big)  =p\,T\in p^{j-2i+1} \,\Herm_m(\OK).$$
\end{proof}

We note various `scaling relations' among the cycles $\ZZ_{i,j}(\xx)$. First, there is an isomorphism
\begin{equation}\label{scale1}
\ZZ_{i,j}(\xx)(\FF) \isoarrow \ZZ_{0, j-2i}(\xx)(\FF), \qquad (\bar  
M^0,M) \mapsto (\MMbar^0, p^{-i}M).
\end{equation}
Next, note that
\begin{equation}\label{scale2}
\Cal W_{i,j}(\xx) = \Cal W_{i,j}(L) = \Cal W_{0,j}(p^iL).
\end{equation}
These two relations
are simply the translation into the language of lattices of the scaling relations on the level of formal schemes in Remark \ref{scalingfor}. 
Finally, note that the set of lattices $\Cal D_j$ for the hermitian  
space $(C, \hgs{}{})$ coincides with the
set of lattices $\Cal D_0$ for the space $(C, p^j\hgs{}{})$. Both sets are empty if $n$ and $j$ are both odd.

\section{Hermitian lattices}

In this section, we consider the cycle $\ZZ_{i,j}(\xx)$ determined by  
an $n$-tuple $\xx$ of special homomorphisms
with {\it nonsingular} fundamental matrix $T = h(\xx,\xx)\in \Herm_n 
(\kay)$.  For global reasons (cf.~the Introduction),
 the cycle $\ZZ_ 
{i,j}(\xx)$ has empty generic fiber in this case, i.e., $p$ is  
locally nilpotent on  $\ZZ_{i,j}(\xx)$.
We give a description
of $\ZZ_{i,j}(\xx)(\FF)$ as a union of the strata in $\Cal N^0_{2i} 
\times \Cal N_j$ defined in \cite{vollaard}.

Recall \cite{vollaard} that the hermitian space $(C,  \hgs{}{})$ is  
determined up to isomorphism by  its dimension, since
\begin{equation}\label{Codd}
\ord\det(C)\equiv \dim(C)+1\mod 2.
\end{equation}
Let $\xx_0(\bar 1_0)= \big[(x_1)_0(\bar 1_0), \dots, (x_n)_0(\bar 1_0)\big]$  
be an $n$--tuple of vectors in $C$
spanning a lattice $L$, %with $L\subset pL^\vee$,
and let
$T'= \{\xx_0(\bar 1_0),\xx_0(\bar 1_0)\} = p\, T$ be the  
corresponding matrix of inner products.
By Lemma \ref{relVC},
$\ord\det(T) = \ord\det(C)-n$, hence
\begin{equation}\label{CoddT}
\ord\det(T) \text { is odd.}
\end{equation}
The cycle $\ZZ_{i,j}(\xx)$ determined by $\xx$ depends only on the $ 
\OK$-lattice $L$,
and, by Proposition~\ref{proj}, the projection of $\ZZ_{i,j}(\xx)(\FF) 
$ to $\Cal N_j(\FF) \isoarrow \Cal D_j$
is the set
$$\Cal W_{i,j}(L) = \{\  A\in \Cal D_j\mid p^iL\subset p^{j+1}A^\vee\  
\}.$$

We first note that $\Cal W_{i, j}(L)$ is a union of strata $\Cal V(\L)$.
\begin{prop}
$$\Cal W_{i,j}(L) = \bigcup_{\substack{ \L\\ \snass p^iL\,\subset\, p^ 
{j+1}\L^\vee}} \Cal V(\L),$$
where the $\L$'s are vertices of level $j$.
\end{prop}
%\marginpar{We had missed the fact that $L\subset pA^\vee$ and not  just $L\subset A$ is required!}
\begin{proof}  The lattice $\L(A)$ associated to $A\in \Cal D_j$ by  
Lemma~\ref{zink} is the
smallest $\tau$-stable $W$-lattice containing $A$. By duality, $p^{j 
+1}\L(A)^\vee$ is the largest
$\tau$-stable $W$-lattice contained in $p^{j+1} A^\vee$. Thus, $p^iL 
\subset p^{j+1}A^\vee$ if and only if $p^iL\subset p^{j+1}\L(A)^\vee$.
Thus, $A\in \Cal W_{i,j}(L)$ if and only if $\Cal V(\L(A))\subset  
\Cal W_{i,j}(L)$.
\end{proof}

%In fact, one can show that $L\subset A$ implies that $L\subset p A^ \vee$.
%\marginpar{\it Omit this remark?}

%Next we have the following result about the equidimensionality of $ \Cal W(L)$.

Our main results about the structure of $\Cal W_{i,j}(L)$ are the  
following. In the rest of the section, $\tilde T=p^{2i-j} T$ will denote the corresponding scaled fundamental matrix. 

\begin{theo}\label{propdim} (i) If
$\tilde T \notin \Herm_n(\OK)$, then $\Cal W_{i,j}(L)$ is empty.\hfb
(ii) If
$\tilde T \in \Herm_n(\OK)$, let $\red (\tilde T)$ be the image of $\tilde T$ in $\Herm_n(\F_{p^2})$,
and let $t_0=t_0(\tilde T)$ be the largest odd integer less than or equal to  
$n - \text{\rm rank}(\,\red(\tilde T)\,)$.
Then
$$\Cal W_{i,j}(L) = \bigcup_{\substack{ \L\\ \snass p^iL\,\subset\, p^ 
{j+1}\L^\vee \\ \snass t(\L)=t_0}} \Cal V(\L).$$
\end{theo}
By Theorem \ref{alggeo}, (ii) we deduce from this theorem:
\begin{cor}
 If it is non-empty, $\Cal W_{i,j}(L)$  is the set of $\FF$-points of a variety of  
pure dimension $\frac{1}{2}(t_0(\tilde T)-1)$.\qed
\end{cor}
\begin{rem}\label{remalggeo}
   Note that $t_0$ only depends on the span $L$ of the components of $ 
\xx_0(\bar 1_0)$.
\end{rem}

\begin{theo}\label{propirred}
Let\footnote{Since $p$ is odd, every $\GL_n(\OK)$-orbit in $\Herm_n(\OK)$ has a unique representative 
of this form.
%Here $\diag(A_0, A_1, \dots, A_k)$ denotes a matrix with 
%the give square matrices $A_i$ along its diagonal and all other entries $0$.
}
$$\tilde T \simeq \diag(1_{n_0},p 1_{n_1}, \dots,  p^k 1_{n_k})$$
be a Jordan decomposition of $\tilde T$ and let
$$n^+_{\text{\rm \ even}} = \sum_{\substack{i \ge 2\\ i\ \text{\rm \  
even}}} n_i\qquad\text{
and}\qquad
n^+_{\text{\rm \ odd}} = \sum_{\substack{i \ge 3\\ i\ \text{\rm \  
odd}}} n_i.$$
Then $\Cal W_{i,j}(L) = \Cal V(\L)$ for a unique vertex $\L$ of level  
$j$ and type $t_0(\tilde T)$ if and only if
\begin{equation*}
\max(n^+_{\text{\rm \ even}}, n^+_{\text{\rm \ odd}}) \le 1. \tag{$\star$}
\end{equation*}
\end{theo}
By Theorem \ref{alggeo}, (ii) we deduce from this theorem:
\begin{cor}
  $\Cal W_{i,j}(L)$ is the set of $\FF$-points of an   irreducible variety  
if and only if the condition $(\star)$ in the previous theorem is satisfied.\qed
\end{cor}
\begin{cor}\label{point}
  $\Cal W_{i,j}(L)$ consists of a single point if and only if
$$n-\text{\rm rank}(\red(\tilde T)) \le 2.$$\qed
\end{cor}

\begin{rem}  Is it true that $\Cal W_{i,j}(L)$ is always connected?
\end{rem}

Before proving these results, we make a few simple observations.
First of all, since $\Cal W_{i,j}(L) = \Cal W_{0,j}(p^iL)$, it  
suffices to consider the case where $i=0$.
Second, for a lattice $A\subset C$,  note that
$p^j A^\vee$ is the dual lattice with respect to the scaled hermitian  
form $p^{-j}\hgs{}{}$.
Thus, if we denote by $C^{(j)}$ the scaled hermitian space $(C, p^{-j} 
\hgs{}{})$, we have
$\Cal D_j = \Cal D_0^{(j)}$ in the obvious notation.  Thus,
$$\Cal W_{i,j}(L) = \Cal W_{0,0}^{(j)}(p^iL).$$
Moreover, $\L$ is a vertex of level $j$ in $C$ if and only if $\L$ is  
a vertex of level $0$ in $C^{(j)}$.
Finally, note that if $T = h(\xx,\xx)$, then $p^{2i-j}T = p^{-j}h(p^i 
\xx,p^i\xx)$.  Thus it suffices to
consider the case where $i=j=0$. It is important to note that for $n$  
odd, $\Cal N_j(\FF)$ is empty unless
$j$ is even. Thus, in all cases, the space $C^{(j)}$ again satisfies  
the parity relation (\ref{Codd}).

 From now on, in this section, we assume that $i=j=0$ and write $\Cal  
W(L)$ for $\Cal W_{0,0}(L)$.
Also, all vertices will have level $0$,
and we assume that $\tilde T=T= h(\xx,\xx) \in \Herm_n(\OK)$.

First note that if $L\subset p\L^\vee$ for some vertex $\L$, then the  
inclusions
$$L\subset p\L^\vee\overset{t}\subset \L \subset p L^\vee,$$
show that $\L/p\L^\vee$ is a subquotient of $pL^\vee/L$ and, in  
particular, the type $t=t(\L)$
is constrained by the structure of $pL^\vee/L$.
More precisely, let $D=D(L) = pL^\vee/L$ with $\kay/\OK$-valued  
hermitian form
determined by $h(\ , ) = p^{-1} \hgs{}{}$. Note that $pL^\vee$ is the  
dual lattice of $L$
with respect to $h$, so that the resulting hermitian form on $D$ is  
nondegenerate.  Below we identify  $\OK/p\OK$ with $\F_{p^2}$. 
\begin{lem}\label{Dform}
(i) Let
$$m=\dim_{\F_{p^2}} D[p] = \dim_{\F_{p^2}} D/pD.$$
Then $m = n - \text{\rm rank}(\red(T))$. \hfb
(ii) There is a bijection between the sets
$$\Ver(L):=\left\{ \L \mid  \text{$\L$ a vertex with $L\subset p\L^ 
\vee$}\right\}$$
and
$$GrD:=\left\{ B \mid \text{$B$ an $\OK$-submodule of $D$ with $pB 
\subset B^\perp \subset B$}\right\}, $$
given by $\L\mapsto \L/L$.
The type $t(\L)$ of $\L$ is the dimension of  the   $
\F_{p^2}$-vector space $B/B^\perp$.
\qed
\end{lem}

\begin{example}\label{00case}
   Suppose that $T \simeq \diag(1_{n_0}, p\,1_{n_1})$, so that $D=D(L)$ is an  
$\F_{p^2}$-vector
space of dimension $n_1$. Then $GrD$ can be identified with the set  
of all
isotropic subspaces $U$ in $D$ via the map $B\mapsto U=B^\perp$.  
Thus, there is a unique
maximal vertex $\L$ in $\Ver(L)$ of type $n_1$ corresponding to $U=0 
$, and $\Cal W(L) = \Cal V(\L)$,
as asserted in Theorem~\ref{propirred}.
\end{example}

\begin{proof}[Proof of Theorem~\ref{propdim}]

\begin{lem}\label{biggertype}
Suppose that $\L\in \Ver(L)$ with
  $\dim_{\F_{p^2}} (\L^\vee\cap pL^\vee)/\L \ge 2$.
Then there exists a lattice $\L_1\in \Ver(L)$ with $\L\subset \L_1$
  and $t(\L_1)>t(\L)$.
\end{lem}
\begin{proof}
Note that the $\F_{p^2}$-vector space $\L^\vee/\L$ has a non- 
degenerate $\kay/\OK$-valued hermitian
form determined by $\hgs{}{}$.  If $\dim_{\F_{p^2}} (\L^\vee\cap pL^ 
\vee)/\L \ge 2$, then
this subspace contains an isotropic line $\ell$.  Let $\L_1 = \pr^{-1} 
(\ell)$
where $\pr: \L^\vee \rightarrow \L^\vee/\L$.  By construction, $\L 
\subset \L_1\subset \L^\vee\cap pL^\vee$,
so that, in particular, $L\subset p\L_1^\vee$ and $p\L_1^\vee\subset p 
\L^\vee\subset \L\subset \L_1$.
Also, since $\ell\subset \ell^\perp$ in $\L^\vee/\L$, we have $\L_1 
\subset \L_1^\vee = \pr^{-1}(\ell^\perp)$.
Thus $\L_1\in \Ver(L)$ and the index $t_1$ of $p\L_1^\vee$ in $\L_1$  
is strictly larger than $t$.
\end{proof}

Next, we record a few more general facts.
For any $\L\in \Ver(L)$,
let $B\in Gr D(L)$ be the $\OK$-submodule of
$D=D(L)$ associated to $\L$. Since the image of $p\L^\vee$ in $D$ is  
$B^\perp$,  the image of $\L^\vee\cap pL^\vee$ in $D$ is
$$\{\, x\in D\mid px\in B^\perp\} = (pB)^\perp.$$
In particular, note that $D[p]\subset (pB)^\perp$ and that the  
quotient $(pB)^\perp/B^\perp$ is killed by $p$.
Since the pairing on $D$ induced by $h$ is perfect,  we have
$$pB\subover{r} B^\perp\subover{t}B\subover{r}(pB)^\perp,$$
where $r=\dim (\L^\vee\cap pL^\vee)/\L$.  But the inclusion on the  
right implies that
the subspace $B[p]\subset D[p]$ has codimension at most $r$, so that  
we obtain
\begin{equation}\label{squeeze}
m\ge t+r = \dim B/pB =\dim B[p] \ge m-r.
\end{equation}
This gives
$$2r \ge m-t = m-t_0 + (t_0-t).$$

\begin{lem}\label{tight} If $\L\in \Ver(L)$ with $t(\L)<t_0$, then
either
$$\dim_{\F_{p^2}}( \L^\vee\cap pL^\vee)/\L \ge 2,$$
or the special case \hfb
{\rm ($\star\star$)}\qquad  $m=t_0$, $t= m-2$, $r=1$,\  and $\ \dim B[p]=m-1$\hfb
holds. 
\end{lem}

\begin{proof}
By assumption, $t_0-t\ge 2$ is even so that $r\ge 2$, as claimed,  
except in the special case.
\end{proof}

\begin{lem} In the special case $(\star\star)$, the lattice $\L_1 = \L^\vee 
\cap pL^\vee$
is in $\Ver(L)$ with $t(\L_1)=t_0$.
\end{lem}
\begin{proof}
In this case,  we have the picture.
$$pB\subover{1} B^\perp\subover{m-2}B\subover{1}(pB)^\perp.$$
For any $x_0\in pL^\vee$ whose image $\bar x_0$ in $D$ lies in $D[p] 
\setminus B[p]$, we have $(pB)^\perp = B + \OK\cdot \bar x_0$.
Thus,  $\L_1 = \L^\vee\cap pL^\vee = \L + \OK x_0$.
As in the proof of Lemma \ref{biggertype}, it follows that $L\subset p 
\L_1^\vee$ and $p\L_1^\vee\subset p\L^\vee\subset \L\subset \L_1$.
But now
$$\hgs{\L_1}{\L_1} = \hgs{\L^\vee\cap pL^\vee}{\L+\OK x_0} \subset \OK$$
because $\hgs{\L^\vee}{\L} = \OK$ and
$$\hgs{pL^\vee}{x_0} = p\, h(pL^\vee, x_0) \subset h(pL^\vee,L)\subset 
\OK,$$
since $p x_0\in L$.  Thus $\L_1\subset \L_1^\vee$, so that $\L_1$ is  
in $\Ver(L)$, as claimed.
\end{proof}
The previous lemmas show that every $\L\in \Ver(L)$ is contained in  
some $\L_0\in \Ver(L)$ with $t(\L_0)=t_0$.
On the other hand, by Lemma~\ref{Dform}, the type of any lattice $\L 
\in \Ver(L)$ is at most $t_0$.
This completes the proof of Theorem~\ref{propdim}.
\end{proof}

\begin{cor}
Suppose that $\L\in \Ver(L)$ with $t_0=t(\L)$ maximal and let $B\in  
Gr D(L)$ be the associated $\OK$-submodule.
Then
$$r=\dim_{\F_{p^2}} (\L^\vee\cap pL^\vee)/\L = \begin{cases} 0&\text 
{if $m$ is odd,}\\
\nass
1&\text{if $m$ is even.}
\end{cases}
$$
Moreover, $p^{-1}L\cap pL^\vee \subset \L$.
\end{cor}
\begin{proof}  If $m$ is odd, then $t=t_0=m$ in (\ref{squeeze}), so  
that $r=0$, while, if $m$ is even, then
$t=t_0=m-1$ and (\ref{squeeze}) forces $r=1$.   The last assertion  
follows from the fact that $B[p]=D[p]$, since both have dimension $m$. 
\end{proof}

\begin{proof}[Proof of Theorem~\ref{propirred}.]
Suppose that $T$ has the given Jordan decomposition with respect to  
some basis $e_1, \dots, e_n$ of $L$,
and note that $pL^\vee$ has basis $p^{-a_i}e_i$, where $h(e_i,e_i)= p^ 
{a_i}$.  If  $\max(n^+_{\text{\rm \ even}}, n^+_{\text{\rm \ odd}})\ge 2$, i.e., if
\begin{equation}\label{excess}
n^+_{\text{\rm \ even}}=\sum_{\substack{i\ge 2 \\ \snass i \text{\rm  
\ even}}} n_i \ge 2 \qquad
\text{(resp. }\quad
n^+_{\text{\rm \ odd}}=\sum_{\substack{i\ge 3 \\ \snass i \text{\rm \  
odd}}} n_i \ge 2 \ ),
\end{equation}
we can scale the $e_i$'s to a basis $f_1, \dots, f_n$ of $C$ for  
which the hermitian form $h$
has matrix
$$ T_0 = \diag(1_{n'},p\,1_{n''}, p^a \,1_2)$$
with $a=2$ (resp. $a=3$).
Let $L_0$ be the lattice spanned by $f_1, \dots, f_n$, and note that
$$pL_0^\vee =[ f_1, \dots, f_{n'}, p^{-1} f_{n'+1}, \dots, p^{-1} f_ 
{n'+n''}, p^{-a} f_{n-1}, p^{-a} f_n].$$
  Also note that $L\subset L_0$ so that $\Ver(L_0)\subset \Ver(L)$.
Take $u\in \Z_{p^2}^\times$ with $uu^\s=-1$, and let
$$g_1 = f_{n-1}+uf_n, \qquad g_2 = f_{n-1}-uf_n.$$
These are isotropic vectors in $L_0$ with $h(g_1,g_2) = 2 p^a$.
Let
$$\L_1 = [f_1, \dots, f_{n'},p^{-1} f_{n'+1}, \dots, p^{-1} f_ 
{n'+n''}, p^{-1}g_1, p^{-a} g_2],$$
and
$$\L_2 = [f_1, \dots, f_{n'},p^{-1} f_{n'+1}, \dots, p^{-1} f_ 
{n'+n''}, p^{-a}g_1, p^{-1} g_2].$$
Then
$$p\L_1^\vee = [f_1, \dots, f_{n'}, f_{n'+1}, \dots,  f_{n'+n''},  
g_1, p^{1-a} g_2],$$
and
$$p\L_2^\vee= [f_1, \dots, f_{n'}, f_{n'+1}, \dots,  f_{n'+n''}, p^{1- 
a}g_1,  g_2],$$
so that $\L_1$ and $\L_2$ are vertices in $\Ver(L_0)$ of type $n''+2$.
Suppose that $\L_1$ and $\L_2$ were contained in a common vertex $\L 
\in \Ver(L)$.
Then we would have
$$\L_1 \subset \L \subset \L^\vee \subset \L_1^\vee,
\qquad\text{and}\qquad \L_2\subset \L \subset \L^\vee\subset  \L_2^ 
\vee,$$
and hence $h(\L_1,\L_2) \subset p^{-1}\OK$.  But $h(p^{-a}g_2,p^{-a} 
g_1) = 2\,p^{-a}$, so this is not the case.
Thus there is more than one vertex $\L\in \Ver(L)$ with $t(\L)=t_0$  
when (\ref{excess}) holds.

Now we prove the converse. Suppose that 
$\max(n^+_{\text{\rm \ even}}, n^+_{\text{\rm \ odd}})\le 1$.  
Then there are several cases for the Jordan decomposition of $T$. 
First suppose that $n^+_{\text{\rm \ even}}=n^+_{\text{\rm \ odd}}=1$, so that $T$ has Jordan decomposition
$\diag(1_{n_0}, p 1_{n_1}, p^a, p^b)$ with $2\le a<b$ and $a+b$ odd.   
Thus, $L= [e_1,\dots,e_n]$ and
$$pL^\vee = [e_1, \dots, e_{n_0}, p^{-1}e_{n_0+1}, \dots, p^{-1}e_ 
{n-2}, p^{-a} e_{n-1},  p^{-b}e_n].$$
Recall that,  by  (\ref{CoddT}),  $\ord \det (T)$ is odd. Thus  $n_1$ must be even, $t_0 = n_1+1$, and
any $\L\in \Ver(L)$ with $t(\L)=t_0$ contains the lattice
$$p^{-1}L\cap pL^\vee = [e_1, \dots, e_{n_0}, p^{-1}e_{n_0+1}, \dots,  
p^{-1}e_{n-2}, p^{-1} e_{n-1},  p^{-1}e_n].$$
 Let $L'=[e_{n-1}, e_n]$, a lattice in the two-dimensional hermitian vector space $V'$ spanned by $e_{n-1}, e_n$. Then the map 
 $$
 \Lambda\mapsto \Lambda\cap V'
 $$
 gives a bijection between the lattices $\Lambda\in \Ver(L)$ with $t(\Lambda)=t_0$, and the lattices $\Lambda'\in\Ver(L')$ with $t(\Lambda')=1$.

 Thus, we may assume that $n=2$ and that $T=\diag(p^a,p^b)$ where $2\le a,\, b$ and $a$ is even. 
 We proceed by  
explicit computation.
Suppose that $L=[e_1,e_2]$ and write $p\L^\vee = [e_1,e_2]S$ for $S 
\in \GL_2(\kay)$
unique up to right multiplication by an element of $\GL_2(\OK)$.
Then we have
$\L  = [e_1,e_2] T^{-1}{}^t\bar{S}^{-1}$ and
$pL^\vee=[e_1,e_2]T^{-1}$,
and the various inclusions amount to the following conditions:
$$\begin{matrix}
L&\subset&p\L^\vee & \iff  &S^{-1}&\in M_2(\OK)\\
\nass
\nass
p\L^\vee&\subset&\L &\iff  & {}^t\bar{S}TS&\in M_2(\OK)\\
\nass
\nass
\L&\subset&\L^\vee&\iff & p S^{-1} T^{-1}{}^t\bar{S}^{-1}&\in M_2(\OK).
\end{matrix}$$
Moreover, $\L$ has type $t(\L) =t_0=1$ if and only if $\ord(\det {}^t 
\bar{S}TS)=1$.
Assuming that this is the case, we may modify $S$ on the right by an  
element of $\GL_2(\OK)$
so that $T_1 := {}^t\bar{S}TS = \diag(1, p)$.
Note that the last of the above conditions is then immediate.
Write
$S = \diag(p^{-a/2},p^{-(b-1)/2}) S_0$.
Then ${}^t\bar{S}_0 T_1 S_0 = T_1$ so that $u=\det(S_0)$ has norm $1$  
and hence is a unit.
After replacing $S_0$ by $S_1=\diag(1,\bar u)S_0$, so that  $\det(S_1) 
=1$, a short calculation shows that
$$S_1= \begin{pmatrix}\a&\b\\ -p^{-1}\bar \b&\bar \a\end{pmatrix}$$
for $\a$ and $\b\in \kay$ with $1=\a\bar\a + p^{-1}\b\bar\b$. Since  
the two terms on the right side of this last identity
have $\ord$'s of opposite parity, we must have $\ord(\a)=0$ and $\ord 
(\b)\ge1$, so that
$S_1$ and $S_0$ lie in $\GL_2(\OK)$.
Thus
$$\L = [p^{-a/2}e_1,p^{-(b+1)/2}e_2]$$
%\marbull%\qquad (\ \text{resp.}\quad \L  = [p^{-(a-1)/2}e_1,p^{-b/2}e_2]\ )$$
is the unique vertex in $\Ver(L)$  with $t(\L)=1$.  % if $a$ is even
%\marginpar{I thought $a$ is even?}
%(resp. $a$ is odd).
%In case $a$ is odd, the analogous reasoning gives $\L = [p^{-(a-1)/2} 
%e_1,p^{-b/2}e_2]$
%as unique vertex in $\Ver(L)$ with $t(\L)=1$.
Of course, this argument just amounts to the fact that the isometry  
group of $T_1$
is anisotropic so that the building of this group reduces to a single  
point.

The cases where $(n^+_{\text{\rm \ even}}, n^+_{\text{\rm \ odd}}) = (1,0)$,  $(0,1)$, or $(0,0)$, 
i.e., where $T=\diag(1_{n_0}, p 1_{n_1}, p^a)$ or $T=\diag(1_{n_0}, p 1_{n_1})$ are easier and will  
be omitted.  The $(0,0)$ case is discussed in Example~\ref{00case}.

This completes the proof of Theorem~\ref{propirred}.
\end{proof}

\section{Intersection multiplicities}

In this section, we fix $i, j$ and consider a non-empty special cycle
$\ZZ_{i,j}(\xx)$, where we assume that $\xx$ is an $n$-tuple of special  
homomorphisms
whose scaled fundamental matrix $\tilde T=p^{2i-j}h(\xx,\xx)=p^{2i-j-1}\hgs{\xx}{\xx}$, which  
is still assumed to be non-degenerate, satisfies
$$ \text{\rm rank}({\rm red}(\tilde T)) \ge n-2.$$
By Corollary \ref{point}, this implies that the cycle $\ZZ_{i, j}(\xx)$ in
$\Cal N_{2i}^0 \times \Cal N_j$ has underlying reduced scheme
of  dimension $0$ which, in fact, reduces to a single point.  Write $ 
\ZZ_{i,j}(\xx) =\Spec R(\xx)$
for a local $W$-algebra $R(\xx)$ with residue field $\F$.
The arithmetic degree  of $\ZZ_{i,j}(\xx)$ is then, by definition, 
$$\degh(\ZZ_{i, j}(\xx)) = \text{length}_W \, R(\xx)\,\cdot \log p,$$
where $\text{length}_W \, R(\xx)$ is the length of $R(\xx)$ as a $W$- 
module.

\begin{theo}\label{theomult} Suppose that $\tilde T$ is $\GL_n(\OK)$- 
equivalent to
$\diag(1_{n-2}, p^a,p^b)$, where $0\le a\le b$. Then $R(\xx)$ is of  
finite length and
$$\degh(\ZZ_{i,j}(\xx)) = \log p \cdot \frac12\,\sum_{l =0}^{a} p^l (a 
+b+1-2l).$$
\end{theo}
Note that $a+b \equiv \ord(\det(T))\!\!\mod 2$ is odd, cf.~(\ref{CoddT}) and the remarks after  
Corollary \ref{point},  so that, in fact, $0\leq a<b$.  As in the previous section,
we may reduce to the case $i=j=0$, which we assume from now on.  
Accordingly we write
$\ZZ(\xx)$ for $\ZZ_{0,0}(\xx)$.

The first step in the proof is to reduce to the case $n=2$.
\begin{lem} \label{redto2}
  Suppose that $\yy = \xx g$ for $g\in \GL_n(\OK)$ has matrix of  
inner products
$h(\yy,\yy) = \diag(1_{n-2}, p^a,p^b)$ where $a$ is even and $b$ is  
odd {\rm (but $a$ and $b$ are not ordered by size)}.
Then $\ZZ(\xx)(\F) =\ZZ(\yy)(\F)$ corresponds to the lattice $A=\tau(A)$
given  by\footnote{Here we have written $y_i$ for $(y_i)_0(\bar 1_0)$, so that $y_i\in C$. }
$$A = Wy_1+\dots+Wy_{n-2}+W p^{-a/2}y_{n-1}+ W p^{-(b+1)/2} y_n.$$

If $(\bold X, \iota,\l_{\bold X}, \rho_{\bold X})$ is the  
corresponding $p$-divisible group, then
there is an isomorphism
$$\underbrace{\YYbar\times\dots\times\YYbar}_{n-1}\times \YY  
\isoarrow \bold X$$
such that, as elements of $\Hom_{\OK}(\YYbar,\bold X)$,
$$\rho_{\bold X}^{-1}\circ y_i =
\begin{cases} \inc_i &\text{if $i\le n-2$,}\\
\nass
\inc_i\circ\Pi^a&\text{if $i=n-1$,}\\
\nass
\,\inc_i\circ \Pi^b&\text{if $i=n$,}
\end{cases}
$$
where $\inc_i$ denotes the inclusion into the $i$-th factor of the  
product. \hfb
Here $\Pi$ denotes the fixed uniformizer in $O_D=\End(\YYbar)$, cf.  
Remark \ref{rem10case}.
\end{lem}
\begin{proof} Let $a=2r$ and $b=2s+1$. For the given lattice $A$, we  
have
$$pA^\vee = Wy_1+\dots+Wy_{n-2}+W p^{-r}y_{n-1}+ W p^{-s} y_n.$$
so that $pA^\vee\subset A$ with index $1$.  Moreover, since $\tau(y_i) 
=y_i$,
we have $\tau(A)=A$ so that $A=\L(A)$ is a vertex and $\ZZ(\xx) =  
\Cal V(\L(A))$ as claimed.
For convenience, we let $u_i=y_i$, for $i\le n-2$, $u_{n-1}=p^{-r}y_ 
{n-1}$ and
$u_n=p^{-s-1}y_n$. The Dieudonn\'e module $M= M_0+M_1$ associated to $A$
in Proposition \ref{inken1} has
$M_0=A$ and $M_1= Wv_1+\dots+W v_n$, where
$v_i= F^{-1}u_i$, for $i\le n-1$ and $v_n = p F^{-1}u_n$. Here recall  
that $FM_1= pA^\vee$.
Then, since
$$\gs{u_i}{Vu_j} = \gs{u_i}{Fu_j} =\delta\,\hgs{u_i}{u_j} = \delta\, 
\delta_{ij}\,\begin{cases} p&\text{if $i\le n-1$,}\\
1&\text{if $i=n$,}
\end{cases}
$$
we have $\gs{u_i}{v_j} = \delta\,\delta_{ij}$ for all $i$ and $j$.

Recalling that $\YYbar$ has Dieudonn\'e module $\MMbar^0 = W\bar1_0+W  
\bar1_1$, with
$F\bar1_0=p\bar1_1$, $F\bar1_1=\bar1_0$ and $\gs{\bar1_0}{\bar1_1}^0= 
\delta$, we see that
$$M = \oplus_{i=1}^n (W u_i+Wv_i) \ \simeq\ \underbrace{ \MMbar^0 
\oplus\dots\oplus \MMbar^0}_{n-1}\oplus\, \MM^0$$
as polarized Dieudonn\'e modules. The inclusion maps are then given by
$$
\inc_i: W \bar1_0+W\bar1_1 \lra M, \qquad \bar1_0 \mapsto u_i, \quad  
\bar1_1\mapsto v_i,
$$
for $ i\le n-1$, resp.
$$
\inc_i: W 1_0+W1_1 \lra M\qquad 1_0\mapsto v_n, \quad 1_1 \mapsto u_n,
$$
for $i=n$.

Finally, recalling that we have already identified the isocrystal of $ 
\bold X$ with that of $\X$
via $\rho_{\bold X}$, we see that the morphism $y_i: \YYbar  
\rightarrow \bold X$  does indeed yield the given elements
$y_i = y_{i0}(\bar1_0)$ of $N_0$. On the other hand, $y_{i1}F=Fy_{i0} 
$, so that $p y_{i1}(\bar1_1) = y_{i1}(F\bar1_0) = Fy_{i0}(\bar1_0)
=F y_i$. Hence $y_i: \MMbar^0\rightarrow M$ is given by
$$\bar1_0\mapsto y_i=\begin{cases} u_i&\text{if $i\le n-2$,}\\p^r u_i& 
\text{if $i=n-1$,}\\
p^{s+1} u_i&\text{if $i=n$,}
\end{cases}\qquad \bar1_1\mapsto \begin{cases} v_i&\text{if $i\le n-2 
$,}\\p^r v_i&\text{if $i=n-1$,}\\
p^{s} v_i&\text{if $i=n$.}\end{cases}$$
Thus, for $i=n$, we have
$$y_n = p^s\,\inc_n\circ \Pi$$
where   $\Pi: \bar1_0 \mapsto p 1_1$, $\bar1_1\mapsto 1_0$ and  
is $W$-linear.
The cases $i\le n-1$ are clear.
\end{proof}

Let $\text{\rm Def}\,(\bold X, \iota,\l_{\bold X}; \xx)$ be the universal formal deformation ring of 
the collection $(\bold X, \iota,\l_{\bold X}; \xx)$, where $\xx$ is our given $n$-tuple of special homomorphisms $x_i: 
\YYbar\rightarrow \bold X$. 
Since, as in the previous lemma,  it is equivalent to deform the linear  
combination $\yy=\xx g$ for
$g\in \GL_n(\OK)$, we have
\begin{equation}\label{defring}
\text{\rm Def}\,(\bold X, \iota,\l_{\bold X}; \xx)=
\text{\rm Def}\,(\bold X, \iota,\l_{\bold X};\yy).
\end{equation}
We now want to calculate the length of this deformation ring.

If $(X,\iota,\l_X;\yy)$ is any deformation, the element $$e_0 =\sum_ 
{i=1}^{n-2} y_i\circ y_i^*$$
is an idempotent in $\End_{\OK}(X)$ so that $X=e_0X\times (1-e_0)X$,  
where
$e_0X$ (resp. $(1-e_0)X$) is a deformation of
$e_0\bold X$ (resp. of  $\bold X':=(1-e_0)\bold X$). Furthermore,   
the polarization decomposes into the product of polarizations of $e_0X 
$ and $(1-e_0)X$. But the collection $[y_1,\dots,y_{n-2}]$
defines an isomorphism $\bar Y^{n-2} \isoarrow e_0X$, compatibly with  
polarizations.
Thus, the deformations of $(\bold X, \iota,\l_{\bold X};\yy)$
are in bijection with those of $(\bold X', \iota,\l_{\bold X'};\yy') 
$, where
$\bold X'\simeq \YYbar\times\YY$, and where $\yy'=[y_{n-1},y_n]$.
Hence we have
$$\text{\rm Def}\,(\bold X, \iota,\l_{\bold X};\yy) =
\text{\rm Def}\,(\bold X', \iota,\l_{\bold X'};\yy').$$

Thus, it suffices to compute the length of the deformation ring (\ref 
{defring})
in the case $n=2$, where $\yy=[y_1,y_2]$.
By the previous lemma, we have
\begin{equation}\label{boldX}
\bold X = \YYbar \times \YY, \ \ \ y_1 = \inc_1\circ \Pi^a, \ \ \  y_2=\inc_2\circ\Pi^b,
\end{equation}
 where $a$ is  
even and $b$ is odd (but $a$ and $b$ are not ordered by size).

Let $\Cal M$ denote the universal deformation space of $(\bold X,  
\iota, \lambda_{\bold X})$. Then $\Cal M\simeq \Spf\  W[[t]]$ and the
locus $\ZZ(y_1)$ (resp. $\ZZ(y_2)$) where $y_1$ (resp. $y_2$)
deforms is a (formal) divisor on this $2$-dimensional regular  
(formal) scheme. The problem is now to determine the length
\begin{equation}\label{deforminter}
\text{\rm length}_W \text{\rm Def}(\bold X, \iota,\l_{\bold X};\yy) =  
\ZZ(y_1)\cdot \ZZ(y_2)
\end{equation}
of the intersection of these two formal subschemes of $\Cal M$.

This problem is solved in section~\ref{sectionmultcomp} as an  
application of a variant of Gross's theory of quasicanonical
liftings described in the next two sections.

\section{Quasi-canonical liftings}

In this section, we first explain a general construction  of  
extending the
endomorphism ring of a $p$-divisible group. We then apply this
construction to quasi-canonical lifts, and show how this can be used to
construct liftings of  $(\bold X, \iota, \lambda_{\bold X})$ to  
finite extensions of $W$.

For a $p$-divisible group $X$, we define the $p$-divisible group $\OK 
\tt X$ in the standard way as an exterior tensor product,
e.g.~\cite{DM}, p.~\!131.
Explicitly, after choosing a $\Z_p$-basis ${\bf e}= (e_1, e_2)$ of $\OK$, we  
define  $\OK\tt X$ to be $X\times X$.
Any other choice $\bf e'$ of a $\Z_p$-basis differs from the first  
choice by a matrix
$g\in \GL_2(\Z_p)$. Then $g$ defines an automorphism $\alpha_{\bf e,  
\bf e'}: X\times X\to X\times X$.
Since $\alpha_{\bf e, \bf e''}=\alpha_{\bf e, \bf e'}\circ\alpha_{\bf  
e', \bf e''}$, we obtain a system of compatible isomorphisms. Hence  
we obtain in an unambiguous way a $p$-divisible group $\OK\tt X$, unique
up to unique isomorphism, with isomorphisms $\beta_{\bf e}: \OK\tt X 
\to X\times X$, for any choice of a basis $\bf e$, such that $\beta_ 
{\bf e'}=\alpha_{\bf e, \bf e'}\circ \beta_{\bf e}$. The action by  
left translation of $\OK$ on itself
defines an action $\OK\to \End(\OK\tt X)$.
We obtain in this way a functor from the category of $p$-divisible  
groups to the
category of $p$-divisible groups with $\OK$-action. It is compatible  
with base change,
$$
(\OK\tt X)\times_SS'=\OK\tt (X\times_SS').
$$
We mention the following properties of this functor.

\begin{lem}\label{OKendos}
  (i) The functor $X\mapsto \OK\tt X$ from the category of $p$-divisible
  groups to the category of $p$-divisible groups with $\OK$-action is  
left adjoint to the functor forgetting the $\OK$-action.\hfb
  (ii) For the Lie algebras
$$
\Lie (\OK\tt X)=\OK\tt_{\Z_p}\Lie\ X.
$$
Similarly, for the $p$-adic Tate modules,
$$
T_p(\OK\tt X)=\OK\tt_{\Z_p}T_p(X).
$$
There is an analogous statement for the Dieudonn\'e module, if $X$ is  
a $p$-divisible group over
a perfect field.\hfb
\noindent (iii) For two $p$-divisible groups $X$ and $Y$,
$$
\Hom(\OK\tt X, \OK\tt Y)\simeq M_2(\Hom (X, Y)),
$$
and
$$
\Hom _{\OK}(\OK\tt X, \OK\tt Y)=\OK\tt_{\Z_p}\Hom (X, Y).
$$
\end{lem}
\begin{proof}
The first two assertions  follow immediately from the definitions.   
For the first
isomorphism in (iii), choose a $\Z_p$-basis of $\OK$ which identifies  
$\OK\tt X$ and $\OK\tt Y$
with $X^2$ and $Y^2$.  For the second isomorphism, note that the left  
hand side is then identified with the
matrices in ${\rm M}_2(\Hom(X, Y))=\Hom (X^2, Y^2)$ that commute with  
the matrix for multiplication by $\delta$.
\end{proof}

Another useful fact is the following.

\begin{lem}\label{directsum}
Let $p\neq 2$. If $(X,\iota)$ is a p-divisible group with $\OK$-action, 
then there is an isomorphism
$$X\times \bar X\isoarrow \OK\tt X$$
given by
$$(z_1,z_2) \mapsto \a_+(1\tt z_1) + \a_-(1\tt z_2),$$
where
$$\a_\pm = \delta\tt1\pm 1\tt\iota(\delta)\in \End(\OK\tt X).$$
Here $\bar X$ denotes the group $X$ with $\OK$-action given by $\iota 
\circ \s$.
This isomorphism is $\OK$-linear, where $\a\in\OK$ acts on $\OK\tt X$  
via $\a\tt1$.
\end{lem}
\begin{proof} We simply observe the following facts. First,
\begin{equation}\label{muswitch}
(\delta\tt 1) \circ \a_{\pm} = \pm \a_{\pm}\circ (1\tt\iota(\delta))  
= \a_{\pm}\circ (\delta\tt 1).
\end{equation}
Also $\a_++\a_- = 2\delta\tt1$ is an automorphism of $\OK\tt X$ with
$\a_+\circ \a_-=0$, and $(\a_{\pm})^2 = (2\delta\tt1)\,\a_\pm$.
\end{proof}
Finally, we will need the following construction of polarizations on $ 
\OK\tt X$.
We consider the standard hermitian form on $\OK$, with $\hgs{1}{1}=1 
$, and the associated
perfect alternating $\Z_p$-valued form
$$
\gs{\alpha}{\beta}=\tr_{\OK/\Z_p} (\alpha\delta^{-1}\beta^\s).
$$
Using this form on the first factor, we have a canonical perfect pairing
\begin{equation}
\langle\ ,\ \rangle:\ (\OK\tt X)\times (\OK\tt X^\vee)\ \to \hat  
{\mathbb G}_m,
\end{equation}
satisfying the identity on points
$$
\langle \alpha\, x,  y\rangle=\langle x,  \alpha^\sigma\, y\rangle .
$$
Using this pairing, we may identify the Cartier dual $(\OK\tt X)^\vee 
$ of $\OK\tt X$ with $\OK\tt X^\vee$.
Hence, starting with a polarization $\lambda: X\to X^\vee$, we obtain  
a $\OK$-linear polarization
$$
{\rm id}_{\OK}\tt\lambda: \OK\tt X\to \OK\tt X^\vee=(\OK\tt X)^\vee,
$$
such that the associated Rosati involution induces the Galois  
automorphism
$\sigma$ on $\OK$.

We now apply this construction to Gross's quasi-canonical lifts.
Let us recall briefly a few facts from Gross's theory, \cite 
{grossqc}, \cite{ARGOS}, that we will use in what follows.

Let $G$ be a  $p$-divisible formal group of height $2$ and dimension  
$1$ over $\FF$. Then $G$ is uniquely
determined up to isomorphism, and $\End (G)=O_D$,
the maximal order in the quaternion division algebra $D$ over $\Q_p$.
After fixing an embedding $i:\OK\hookrightarrow O_D$, $G$ becomes a
formal $\OK$-module of height $2$. There is a unique lifting $F_0$ of  
this
formal $\OK$-module to $W$, the {\it canonical lift}.

\begin{rem}\label{YYvsbarYY}
Note that after fixing the embedding $i$ of $\OK$ into $O_D$ we may
identify $G$  with the
group $\YY$ of the previous sections. Let $\Pi$ be a uniformizer of $D 
$ that normalizes $\OK$ and with $\Pi^2=p$.
Then the embedding $i\circ\s$ of $\OK$ into $O_D$ is the conjugate of  
$i$ by $\Pi$, and the group
$G$ with the embedding $i\circ \s$ is the group $\bar\YY$.
The formal $\OK$-modules  $\YY$ and $\bar\YY$ are not isomorphic\footnote{The  
analogous groups for a ramified
extension $\smallkay$ are isomorphic.} and $\Pi$ gives an isogeny between  
them of
degree $p$.
\end{rem}

For an integer $s\ge1$,  let
$$
\OKs=\Z_p+p^s\OK.
$$
be the order of conductor $s$ in
$\OK$.

A {\it quasi-canonical lift $F_s$ of level $s$} of $G$ is a lifting  
of $G$ to some finite
extension $A$ of $W$ with  endomorphism ring equal
to $\OKs$, such that the induced action of $\OKs$ on $\Lie\ F_s$ is  
through the
embedding $\OKs\subset W\subset A$ and such that the image of $\OKs$  
in $\End(G)$ is
contained in the image of the fixed embedding of $\OK$ in $O_D$, cf.~ 
\cite{wewers1}, 
Definition~3.1\footnote{More precisely, the induced embedding of $\OKs 
$ coincides with the restriction to $\OKs\subset \OK$ of
 $i$, when $s$ is even,  or $i\circ \s$, when $s$ is odd.
This definition differs in fact from that given  in loc.cit, where it  
is required
that the induced embedding of $\OKs$ in $\End(G)$ is equal to the restriction of 
the fixed embedding $i$ of $\OK$ into $O_D$. However, this condition
is too strong since it would not allow
quasi-canonical lifts for odd $s$, cf.  for example Lemma~\ref 
{frobtwist}. }.
  Any quasi-canonical lift of level $s$ is defined over $W_s$, the ring
  of integers in the ring class field $M_s$ of $\OKs^{\times}$,
i.e., the finite abelian extension of  $M=W_{\Q}$ corresponding
to the subgroup $\OKs^{\times}$ of $\OK^{\times}$ under the  
reciprocity isomorphism.
Note that $M_s$ is a totally ramified extension of $M$ of degree
$$
e_s=p^{s-1}(p+1).
$$
Quasi-canonical lifts of level $s$ always exist; they are not unique,  
but any two are conjugate
under the Galois group $\Gal(M_s/M)$, and, in fact, this Galois group  
acts in a simply transitive way
on all quasi-canonical lifts of level $s$. Quasi-canonical lifts are  
isogenous to the canonical lift.
More precisely, there exists an isogeny $\psi_s: F_0\to F_s$ of  
degree $p^s$ defined over $W_s$.
In terms of the Tate modules of the generic fibers, the isogeny  $ 
\psi_s$ can be described as follows.
There exists a generator $t$ of the free $\OK$-module  $T_p(F_0)$,  
such that
$$
T_p(F_s)=(\Z_p\cdot p^{-s}+\OK)\cdot t.
$$
We note that when $s$ is even, the isogeny $\psi_s$  is  compatible  
with the embedding
of $\OK$  in $O_D=\End (G)$, whereas when $s$ is odd, it is  
compatible with $\iota\circ \s$.  More precisely,
\begin{lem}\label{frobtwist}
Let $\Pi$ be a uniformizer  of $O_D$ which normalizes $\OK$, as above.
Given a quasi-canonical lift $F_s$, there exists a unique 
$\OK$-linear isogeny
  $\psi_s:F_0\to F_s$ which induces the endomorphism $\Pi^s$ on the  
special fiber $G$.
  Moreover, the set $H_{0,s}\subset O_D$ of homomorphisms from $F_0\tt 
\F=G$
  to $F_s\tt \F=G$ that lift to homomorphisms from $F_0$ to $F_s$ is  
precisely
  $\Pi^s\OK$. \qed
\end{lem}

We will refer to the canonical lift $F_0$ as a 
quasi-canonical lift of level $0$.

Let $X^{(s)}=\OK\tt F_s$, for some quasi-canonical lift $F_s$ of  
level $s$. Then $X^{(s)}$ is a
$p$-divisible group with $\OK$-action $\iota$ over $W_s$, with $\Lie\  
X^{(s)}=\OK\tt_{\Z_p}\Lie\ F_s$,
hence $X^{(s)}$ satisfies the signature condition $(1,1)$. Let $\lambda$
be a $p$-principal polarization of $F_s$. Note that $\lambda$ is  
unique up to a scalar in $\Z_p^\times$,
since it induces  a perfect alternating form on the $2$-dimensional $ 
\Z_p$-module $T_p(F_s)$.
By the construction outlined above, we obtain  a $p$-principal $\OK$- 
linear
polarization $\lambda^{(s)}={\rm id}_{\OK}\tt\lambda$ on $X^{(s)}$  
for which the
Rosati involution induces the Galois conjugation  on $\OK$. The special
fiber  of $X^{(s)}$ is equal to $\OK\tt (F_s\tt_{W_s}\FF)$. Now $F_s 
\tt_{W_s}\FF$ is
equal to $G$; however, as a formal $\OK$-module it is equal to $\YY$  
when $s$ is
even, and equal to $\YYbar$ when $s$ is odd. Applying  Lemma \ref 
{directsum}, we obtain identifications
\begin{equation}\label{twistid}
X^{(s)}\tt_{W_s}\FF\ =\
\begin{cases}\  \YY\times\YYbar&\text{$s$ even }\\
\  \YYbar\times\YY&\text{$s$ odd .}
\end{cases}
\end{equation}
Recall from (\ref{boldX}) that $\bold X=\YYbar\times\YY$. We now  
identify $X^{(s)}\tt_{W_s}\FF$ with
$\bold X$ by using the switch isomorphism $\YY\times\YYbar\simeq  
\YYbar\times\YY$
when $s$ is even, resp. the identity map  on $\YYbar\times\YY$ when $s 
$ is odd.
In all cases we have therefore obtained  canonical $\OK$-linear  
isomorphisms
\begin{equation}\label{unnatural}
\rho^{(s)}: X^{(s)}\tt_{W_s}\FF\simeq \bold X.
\end{equation}
By pulling back the polarization $\l_{\bold X}$ to $X^{(s)}\tt_{W_s} 
\FF$, we obtain a
$\OK$-linear $p$-principal polarization on $X^{(s)}\tt_{W_s}\FF$  
which differs from
$\l^{(s)}$ by a scalar in $\Z_p^\times$ (same argument as above,  
using the
Dieudonn\'e module instead of the Tate module). Hence we may change $ 
\l^{(s)}$
by this scalar such that these polarizations coincide. Hence $(X^ 
{(s)}, \iota, \l^{(s)})$ is
a deformation of $(\bold X, \iota, \l_{\bold X})$ to $W_s$. We  
therefore obtain a morphism
\begin{equation}
\varphi_s: \Spf \ W_s\to \Cal M.
\end{equation}
\begin{lem}
The morphism $\varphi_s$ is a closed immersion.
\end{lem}
\begin{proof}
Denoting by $R$ the affine ring of $\Cal M$, we have to show  that the
morphism $\varphi_s^*: R\to W_s$ is surjective. Now $R\simeq W[[t]]$, and $W_s$ is a finite totally ramified extension 
of $W$. To show that $\varphi^*: R\to W_s$ is surjective,  it therefore suffices to show that $\varphi^*(t)$ is a uniformizing parameter, i.e., that
$W_s\tt_R \F=\F$. 
 By the universal property of $\Cal M$, this says  
that the locus in
$\Spec\ W_s/pW_s$ where there exists an $\OK$-linear isomorphism
$\alpha: \OK\tt F_s\to \OK\tt G$  is equal to $\Spec \FF$. According  
to Lemma~\ref{OKendos}  we can write
$\alpha=1\tt \alpha_0+\delta\tt \alpha_1$, where $\alpha_0$ and $ 
\alpha_1$
are homomorphisms from $F_s$ to $G$. By identifying the special  
fibers of $F_s$ and
$G$, $\alpha$ becomes a unit in $O_D$. But then $\alpha_0$ or $ 
\alpha_1$ is a
unit in $O_D$, and hence one of them defines an isomorphism from $F_s 
\longrightarrow F_0$
over this locus. By \cite{wewers1}, Cor.~4.7, this implies that the  
locus in question is reduced to the special point.
\end{proof}
\begin{defn}\label{defn6.6}
Let $\ZZ_s$ be the divisor on $\Cal M$ defined by the image of $ 
\varphi_s$.
\end{defn}

\section{Deformations of homomorphisms}

We consider the following lifting problem. As in
\cite{wewers1}, suppose that $A$ is a finite extension of $W$ with  
uniformizer $\l$,
and let $A_m = A/\l^{m+1}$. We also let $e$ be the ramification index  
of $A$ over $W$
and denote by $\ord_A$ the discrete valuation on $A$ with $\ord_A(\l) 
= 1/e$.
For integers $r$, $s\ge0$ let $F_r$ and $F_s$
be quasi-canonical liftings defined over $A$. Suppose that a  
homomorphism
$$\mu: (\OK\tt F_0)\tt_W\F \lra (\OK\tt F_s)\tt_{W_s}\F$$
is given. Let
$m_s(\mu)$ be the maximum $m$ such that
$\mu$ lifts to a homomorphism from $(\OK\tt F_0)\tt_W A_m$ to $(\OK 
\tt F_s)\tt_{W_s} A_m$.
Also, as in the first part of (iii) of Lemma~\ref{OKendos},  write
\begin{equation}\label{mumatrix}
\mu= \begin{pmatrix}\mu_1&\mu_2\\\mu_3&\mu_4\end{pmatrix},
\end{equation}
with $\mu_i\in \Hom(F_0\tt_W\F,F_s\tt_{W_s}\F) = O_D$. 

Our aim is to prove the following theorem.

\begin{theo}\label{tlinearThm} Write $\mu$ as in (\ref{mumatrix}),  and
suppose that
$$\mu_i\in (\Pi^s\OK + \Pi^{l_i} O_D) \setminus (\Pi^s\OK + \Pi^{l_i 
+1}O_D)$$
for integers $l_i\ge 0$. Let $l=\min_i\{l_i\}$.
Then
$$m_s(\mu) =  e/e_s\cdot \begin{cases} {\disp \frac{p^{l+1}-1}{p-1}}& 
\text{if $l<s$,}\\
\nass
\nass
{\disp\frac{p^s-1}{p-1}}+ \frac12\,(l+1-s)\,e_s&\text{if $l\ge s$.}
\end{cases}
$$
\end{theo}

\begin{proof}
Note that $\mu$ lifts to $A_m$ if and only if the components  $\mu_i$  
all lift
to homomorphisms $F_0\tt_W A_m \rightarrow F_s\tt_{W_s} A_m$.
Recall from \cite{wewers2}, section 1.4, that for a homomorphism $ 
\psi:  F_r\tt_{W_r}\FF \lra  F_s\tt_{W_s}\FF$,
$n_{r,s}(\psi)$ is defined to be the maximum $m$ such that $\psi$  
lifts to a homomorphism
$F_r\tt_{W_r} A_m \rightarrow F_s\tt_{W_s} A_m$.  Thus,
$$m_s(\mu) = \min_i\{n_{0,s}(\mu_i)\}.$$

So, we are reduced to determining  the quantities $n_{0,s}(\mu_i)$. To this end, we prove the next Proposition, a slight extension of  \cite{wewers2}, Proposition 1.2. 

As in \cite{wewers2}, let $H_{r,s}\subset D$ be the subset of  
elements $\phi$
that lift to homomorphisms from $F_r$ to $F_s$.
For example, $H_{s,s}= \OKs$.
In general, if $s\geq r$, with the conventions introduced in the  
previous section,  $H_{r,s} = \Pi^{s-r}\OKr$ and  there is an  
isomorphism
\begin{equation}\label{scaleHrs}
H_{r,s} \isoarrow H_{r,s+1}, \qquad \phi \mapsto \Pi\psi,
\end{equation}
cf.~\cite{wewers2}, Proposition 1.1, 1). 
Also, passage to dual isogenies shows that $n_{r,s}(\psi)=n_{s,r} 
(\psi^*)$. Hence we may always assume that $s\geq r$, as we shall do  
from now on.
For $\psi\in O_D \setminus H_{r,s}$,
let
\begin{equation}\label{lrsdef}
l_{r,s}(\psi) = \max\{\,v(\psi+\phi)\,\mid \,\phi\in H_{r,s}\,\}
\end{equation}
where $v$ is the valuation on $D$ with $v(\Pi)=1$. More explicitly,
$l=l_{r,s}(\psi)$ is the positive integer such that
\begin{equation}\label{betterlrsdef}
\psi \in (\Pi^{s-r}\OKr + \Pi^lO_D) \setminus  (\Pi^{s-r}\OKr+ \Pi^{l 
+1}O_D).
\end{equation}
Note that, if $v(\psi) < s-r$, then $l_{r,s}(\psi) = v(\psi)$.  
If $v(\psi)\ge s+r$,
then $l_{r,s}(\psi) +r -s \ge 2r$ is odd, cf.~\cite{vollaardARGOS}, Remark 2.2.

\begin{prop}\label{prop6.3} Let $l= l_{0,s}(\psi)$. Then
$$n_{0,s}(\psi) = e/e_s\cdot \begin{cases} {\disp \frac{p^{l+1}-1} 
{p-1}}&\text{if $l<s$,}\\
\nass
\nass
{\disp\frac{p^s-1}{p-1}}+ \frac12\,(l+1-s)\,e_s&\text{if $l\ge s$.}
\end{cases}
$$
\end{prop}
\begin{proof} In fact, we will determine $n_{r,s}(\psi)$ for any $r 
\leq s$. First, by \cite{vollaardARGOS}, Theorem~2.1, if $\psi\in O_D 
\setminus H_{r,r}$ with $l_{r,r}(\psi) = l$, then
$$n_{r,r}(\psi) = e/e_r\cdot
\begin{cases}
2\,{\disp\frac{p^{l/2+1}-1}{p-1}} -p^{l/2}&\text{if $l\le 2r$ is  
even,}\\
\nass
\nass
2\,{\disp\frac{p^{(l+1)/2}-1}{p-1}}&\text{if $l\le 2r$ is odd,}\\
\nass
\nass
2\,{\disp\frac{p^{r}-1}{p-1}}+ \frac12(l+1-2r)e_r&\text{if $l\ge  2r-1 
$.}
\end{cases}
$$

Next, we recall that Lemma~3.6 of \cite{rapo}  is the following (note  
that $e/e_{s+1}=\ord_A(\pi_{s+1})$ where $\pi_{s+1}$ is a uniformizer  
of $W_{s+1}$):
\begin{lem}\label{onestep}
Suppose that $F_r$, $F_s$ and $F_{s+1}$ are defined over $A$
and that $\psi \in O_D\setminus H_{r,s}$ for $r\le s$.
Then
$$n_{r,s+1}(\Pi \psi) = n_{r,s}(\psi) +e/e_{s+1}.$$\qed
\end{lem}

Now suppose that $s>r$ and that $\psi \in O_D \setminus H_{r,s}$ with  
$l=l_{r,s}(\psi)\ge s-r$.
Then, by (\ref{scaleHrs}) and (\ref{betterlrsdef}),  $\Pi^{r-s}\psi  
\in O_D \setminus H_{r,r}$, and,
by Lemma~\ref{onestep},
we have
\begin{equation}\label{nrsformula}
n_{r,s}(\psi) = \frac{e}{e_s}+\frac{e}{e_{s-1}}+ \dots + \frac{e}{e_{r 
+1}} + n_{r,r}(\Pi^{r-s}\psi)\ .
\end{equation}

Next suppose that $s>r$ and that $l=l_{r,s}(\psi)<s-r$. Then we may  
assume that $v(\psi)=l=l_{r,s}(\psi)$. In this case, we may pull out $ 
\Pi^l$ of $\psi$, and obtain
\begin{equation}\label{nrsformula2}
n_{r,s}(\psi) = \frac{e}{e_s}+\frac{e}{e_{s-1}}+ \dots + \frac{e}{e_ 
{s-l+1}} + n_{r,s-l}(\Pi^{-l}\psi)\ .
\end{equation}
Now $\Pi^{-l}\psi\in O_D^\times$, so that a lift of $\Pi^{-l}\psi$  
over $A_m$
defines an isomorphism $F_r \tt A_m \isoarrow F_s\tt A_m$.
\begin{lem} Suppose that $l_{r,s}(\psi)=0$. Then $$n_{r,s}(\psi) = (e/ 
e_s)\cdot e_r.$$
\end{lem}
\begin{proof} This follows from \cite{KRYbook}, Prop. 7.7.7.
\end{proof}

Thus, if $\psi \in O_D \setminus H_{r,s}$ with $l=l_{r,s}(\psi)<s-r$,  
we obtain
$$n_{r,s}(\psi) = \frac{e}{e_s}+\frac{e}{e_{s-1}}+ \dots + \frac{e}{e_ 
{s-l+1}}+(e/e_s)\cdot e_r\ .$$
Since, for $0\le k <s$,  $e_s/e_{s-k} = p^k$ and $e_0=1$, we obtain  
the  expressions claimed in  Proposition~\ref{prop6.3}. Theorem~\ref{tlinearThm} follows immediately.
\end{proof}\end{proof}

\section{Computation of intersection multiplicities}\label 
{sectionmultcomp}

We now return to the situation at the end of section 5.
%and show how to reduce the $n=2$
%deformation problem described there to the deformations of  homomorphisms discussed in the previous section.
%This will allow us to deduce Theorem~\ref{theomult} from Theorem~\ref {tlinearThm}.
%Recall that we have the pair of $\OK$-linear homomorphisms $\yy =  [y_1,y_2]$,
%$y_i:\bar \YY \rightarrow \bold X=\bar\YY\times\YY$, with $y_1=\inc_1 \circ \Pi^a$ and $y_2=\inc_2 \circ \Pi^b$
%with $a$ even and $b$ odd.
%For $i=1$, $2$, let $\ZZ(y_i)$ be the locus in
%$\Cal M$ where $(\bold X, \iota,\l_{\bold X};y)$ deforms.  Then $\ZZ (y_i)$ is a divisor on $\Cal M$.
To compute the intersection number (\ref{deforminter}), we first  
decompose the divisors $\ZZ(y_i)$.

%\end{proof}
%\end{document}

\begin{prop}\label{Zycompo}
As divisors on $\M$
$$\ZZ(y_1) = \sum_{\substack{s=0 \\ \snass\text{\rm \ even}}}^{a} \ZZ_s
\qquad\text{and}\qquad
\ZZ(y_2) = \sum_{\substack{s=1 \\ \snass\text{\rm \ odd}}}^{b} \ZZ_s.$$
Here $\ZZ_s$ is the divisor on $\M$ given in Definition~\ref{defn6.6}. 
\end{prop}

\begin{proof} We begin by showing that each $\ZZ_s$ for $s$ in the  
given range
lies in $\ZZ(y_i)$.  Suppose that $F_s$ is a quasi-canonical lift of  
level $s$. There is then a unique
isogeny $\psi_s:F_0 \rightarrow F_s$ of degree $p^s$ such that the  
reduction
$$\bar\psi_s: F_0\tt_W \F\lra F_s\tt_{W_s} \F$$
is equal to $\Pi^s$. Here recall that, for any quasicanonical lift  
$F_s$, an identification $F_s\tt_{W_s}\F = G$ is given
so that $\bar\psi_s$ is identified with an element of $O_D$.
We also write $\psi_s$ for the corresponding isogeny
$$\psi_s: \Xo \lra \Xs.$$
There is an isomorphism
$$\gamma: \bar Y \times Y \isoarrow \Xo$$
given by composing the isomorphism $Y\times \bar Y \isoarrow \Xo$
of Lemma~\ref{directsum} with the switch of factors.
We then obtain an $\OK$-linear isogeny
$$\psi_s\circ\gamma: \bar Y \times Y\lra \Xs.$$
Note that the diagram
$$
\begin{matrix}
\Xo\tt\F  &  \overset{\bar\psi_s}{\lra} &\Xs\tt\F\\
\nass
\bar\gamma\,\uparrow\,\phantom{\bar\gamma}&{}&\phantom{\rho^{(s)}}\, 
\uparrow\,\rho^{(s)}\\
\nass
\bar\YY\times \YY&\overset{g_o^s}{\lra}& \bar\YY\times \YY
\end{matrix}
$$
is commutative where $g_o = \text{sw}\circ (\Pi\times \Pi)$ where $ 
\text{sw}$ is the switch of factors.
The map
$$\psi_s\circ\gamma\circ \inc_1:\bar Y \lra \Xs$$
is an $\OK$-linear homomorphism whose reduction is
$$\overline{\psi_s\circ\gamma\circ \inc_1} = \begin{cases} \inc_1 
\circ \Pi^s&\text{if $s$ is even,}\\
\nass
\inc_2\circ \Pi^s&\text{if $s$ is odd.}
\end{cases}
$$
Now let
\begin{equation}\label{Zslifts}
\begin{cases}
\tilde y_1 = p^{(a-s)/2}\psi_s\circ \inc_1&\text{if $s$ is even}\\
\nass
\tilde y_2 = p^{(b-s)/2}\psi_s\circ \inc_1&\text{if $s$ is odd.}
\end{cases}
\end{equation}
Then, for $s$ of the correct parity,  $\tilde y_i: \bar Y \rightarrow  
\Xs$
is a lift of $y_i$.
This  shows that the divisor $\ZZ_s$ is a component of $\ZZ(y_1)$  
(resp. $\ZZ(y_2)$ )
for
$0\le s\le a$ even (resp. $1\le s\le b$, odd).
Hence we obtain inequalities of divisors on $\Spec W [[t]]$,
$$\sum\limits_{\substack{s = 0\\{\rm \ even}}}^a \ZZ_s\leq \ZZ (y_1) 
\ ,  \text{ resp. }
\sum\limits_{\substack{s = 1\\{\rm \ odd}}}^b \ZZ_s\leq \ZZ (y_2) .$$
In order to show that these inequalities are equalities, it suffices to
show that the intersection multiplicities of both sides with the special
fiber $\M_p = \Spec k[[t]]$ are the same. For the LHS we obtain for  
these
intersection multiplicities
\begin{equation*}
\sum\limits_{\substack{s = 0\\{\rm \ even}}}^a \ZZ_s\cdot\M_p = \sum 
\limits_{\substack{s = 0\\{\rm \ even}}}^a e_s = 1+ (p+p^2)+\ldots +  
(p^{a-1}+p^a) = \frac{p^{a+1}-1}{p-1}
\end{equation*}
resp.
\begin{equation*}
\sum\limits_{\substack{s = 1\\{\rm \ odd}}}^b \ZZ_s\cdot\M_p = \sum 
\limits_{\substack{s = 1\\{\rm \ odd}}}^b e_s = (1+p)+(p^2+p^3)+ 
\ldots + (p^{b-1}+p^b) = \frac{p^{b+1}-1}{p-1}\ .
\end{equation*}
The assertion now follows from the following proposition.
\end{proof}
Let $y:\overline{\YY}\to \bold X$ be an $\OK$-linear homomorphism with
$y^\ast\circ y\neq 0$.  Consider the  
universal
deformation of $(\bold{X}, \iota , \lambda_{\bold X})$  in
equal characteristic over $\M_p = \Spf\ \FF [[t]]$, and the maximal  
closed formal subscheme
$\ZZ(y)_p$ of $\M_p$, where $y$ deforms into a  
homomorphism
$y :\overline{\YY}\times_{\Spec\FF}\ZZ (y)_p\to X\times_{\M_p}\ZZ (y) 
_p$.
The proof of the following proposition is due to Th. Zink.
\begin{prop}\label{length}
The length of the Artin scheme $\ZZ (y)_p$ is equal to 
$\frac{p^{v+1}-1}{p-1}$, 
where  $v$ is the $D$-valuation of the element
 $y^\ast\circ y\in O_D$ (maximal power of $\Pi$ dividing $y^\ast\circ y$).
\end{prop}
\begin{proof}
  ({\it Zink}): We are going to use the theory of displays \cite 
{displays} and their windows \cite{windows}. Let $R = \FF[[t]]$ and
$A = W[[t]]$. We extend the Frobenius automorphism $\sigma$ on $W$ to  
$A$ by setting $\sigma (t) = t^p$.
For any $a\geq 1$, we set $R_a = \FF[[t]]/t^a$ and $A_a = A/t^a$.  
Then $A$ is a
frame for $R$, resp. $A_a$ is a  frame for $R_a$, with augmentation  
ideal generated
by $p$.

  We consider the category of $p$-divisible groups over $R$ which  
have no \'etale part modulo $t$ or, in other
words, the category of formal groups over $R$ which are $p$-divisible
modulo $t$. For simplicity of expression we call them  formal $p$- 
divisible groups over $R$.

Formal $p$-divisible groups over $R$ are classified by $A\!-\! R$-windows
$(M,M_1,\phi, \phi_1)$, which satisfy a nilpotence
condition, \cite{windows}, Thm. 4.  Recall that
an $A\!-\!R$-window consists of a $4$-tuple $(M, M_1, \phi,  \phi_1)$, where
$M$ is a free $A$-module of finite rank and $M_1$ is a submodule  
containing $p M$
such that $M/M_1$ is a free $R$-module. Furthermore, $\phi : M\to M$  
is a
$\sigma$-linear endomorphism such that $\phi (M_1)\subset p M$ and  
such that
$\phi (M_1)$ generates $p M$ as an $A$-module (this last condition is  
easily
seen to be equivalent to condition (ii) in \cite{windows}, Def. 2).  
Finally
$\phi_1 = \frac{1}{p}\ \phi : M_1\to M$.

There is a base change functor from $A\!-\!R$-windows to
$W\!-\!\FF$-windows. This is compatible with base-changing formal $p$-divisible groups. (To see this, one  first passes 
to the display associated to the window, then applies the base change property  of displays, cf.~\cite{displays}, Definition 20, and then passes  back to the associated window. We are using here that the frames for $R$ and for $\FF$ are chosen in a compatible way.) The  category of $W\!-\!\FF$-windows is isomorphic to the
category of ordinary Dieudonn\'e modules over $k$. The nilpotence
condition says that $V$ is topologically nilpotent on $M_{\FF}$.  
Since we will only consider deformations of
formal $p$-divisible groups, the nilpotence condition is always  
automatically satisfied and we will therefore ignore it.

%\end{proof}
%\end{document}

  We denote by
\begin{equation*}
\phi_1^\sharp : M_1^{(\sigma)}\to M
\end{equation*}
the linearization of $\phi_1$, where $M_1^{(\sigma)} = A\otimes_{A,  
\sigma} M_1$.
\begin{lem}
$\phi_1^\sharp$ is an isomorphism.
\end{lem}
\begin{proof}
Choosing a normal decomposition, we have
\begin{equation*}
M = T\oplus L\quad ,\quad M_1 = p T\oplus L\ .
\end{equation*}
The assertion follows since we see that $\phi_1^\sharp$ induces a  
surjection
between free $A$-modules of the same rank.
\end{proof}
We obtain from $(M, M_1, \phi, \phi_1)$ the free $A$-module $M_1$ and  
the
$A$-linear homomorphism $\alpha : M_1\to M_1^{(\sigma)}$ as the  
composition
\begin{equation*}
\alpha:\ M_1\hookrightarrow M \overset{(\phi_1^{\sharp})^{-1}}{\lra}  
M_1^{(\sigma)}\ .
\end{equation*}
In this way, the category of formal $p$-divisible groups over $R$  
becomes equivalent
to the category of pairs $(M_1, \alpha)$, consisting of a free 
$A$-module of finite
rank and an $A$-linear injective homomorphism $\alpha : M_1\to M_1^ 
{(\sigma)}$,
such that Coker $\alpha$ is an $R$-module which is free. An analogous  
description
holds for the category of formal $p$-divisible groups over $R_a$.
Under this equivalence the category of formal
$p$-divisible groups with an $\OK$-action becomes equivalent to
the category of pairs $(M_1, \alpha)$, such that $M_1$ is
$\mathbb{Z}/2\mathbb{Z}$-graded and $\alpha$ is a homogeneous
morphism of degree 1.

Consider the  $p$-divisible group $\overline{\YY}$ over $\FF$ with its
action $\iota$ of $\OK$. It corresponds to the pair\footnote{In fact, for convenience,  we are writing here $M_1[1]$ (degree shift by $1$)
for the situation at hand. This has no effect on the outcome of the calculation.} $(N, \beta)$,  
where $N$ is the
$\Z /2$-graded free $W$-module of rank $2$ with
\begin{equation*}
N^0 = W(k)\cdot n_0\quad , \quad N^1 = W(k)\cdot n_1
\end{equation*}
and
\begin{equation*}
\beta (n_0) = p\otimes n_1\quad , \quad \beta (n_1) = 1\otimes n_0\ .
\end{equation*}

By base change $W \rightarrow A$ we obtain
the pair $(\mathcal N, \beta)$ over $A$ with the same defining relations  which corresponds  
to the constant $p$-divisible
group $\overline{\mathbb Y}$ over $R$.

The $p$-divisible group $\bold{X}$ over $\FF$ corresponds to the
$A_1$-module $M = N\oplus\overline{N}$, where
\begin{equation*}
M = M^0\oplus M^1\quad\text{ and }
\end{equation*}
\begin{equation*}
M^0 = A\cdot f_0\oplus A\cdot e_0\quad ,\quad M^1 = A\cdot f_1\oplus  
A e_1\ ,
\end{equation*}
and to the graded map $\alpha : M\to M^{(\sigma)}$ given by
\begin{equation*}
\begin{matrix}
\alpha (f_0) = p\otimes f_1 \quad , \ & \alpha (e_0) = 1\otimes e_1\cr
\alpha (f_1) = 1\otimes f_0 \quad , \ & \ \alpha (e_1) = p\otimes e_0\ .
  \end{matrix}
\end{equation*}
We consider the deformation of $(\bold{X}, \iota)$  given by the
free $A$-module $\mathcal M$  with the same generators as for $M$ and the homomorphism
$\alpha_t$ (in terms of the ordered basis $f_0, e_0, e_1, f_1$),

\begin{equation*}
\alpha_t=\begin{pmatrix}
  \text{\LARGE{$0$}} & \begin{matrix} 0 &1 \\
                             p & t
			   \end{matrix}\\
\begin{matrix} 0 &1\\
                p &-t
\end{matrix} &        \text{\LARGE{$0$}}
  \end{pmatrix} =
\begin{pmatrix}
  0 & U\cr
  \breve{U} & 0
  \end{pmatrix}\ .
\end{equation*}
Here
\begin{equation*}
U =
\begin{pmatrix}
  0 & 1\\
  p & t
  \end{pmatrix}\ , \text{  resp.   }\breve{U} =
\begin{pmatrix}
  0& 1\\
  p&- t
  \end{pmatrix}\ .
\end{equation*}
One checks that this deformation respects the polarization $\lambda_ 
{\bold X}$ of $\bold X$ (rewrite the deformation in terms of the  
original display $(M, M_1, \phi, \phi_1)$ and use \cite{mu-ord}, Cor.  
3.29). In fact, $(\mathcal M, \alpha_t)$ defines the universal deformation
of $(\bold X, \iota, \lambda_{\bold X})$, cf.~\cite{mu-ord}.

%\end{proof}
%\end{document}

Now let $y$ correspond to the graded $A_1$-linear homomorphism,
\begin{equation*}
\gamma : N\to M\ .
\end{equation*}
Then the length $\ell$ of the deformation space of $\gamma$ is the
maximal $a$ such that there exists a lift
$$
\tilde{\gamma}: \mathcal N\tt R_a\to\mathcal M\tt R_a\ ,
$$
making the diagram  
\begin{equation*}
\xymatrix{
\mathcal N\tt R_a\ar[r]^\beta\ar[d]_{\tilde\gamma} & \mathcal N^{(\sigma)}\tt R_a\ar[d]^{\tilde 
{\gamma}^{(\sigma)}}\\
\mathcal M\tt R_a\ar[r]^{\alpha_t} & \mathcal M^{(\sigma)}\tt R_a\ .
  }
\end{equation*}
commute. 

To calculate $\ell$, we distinguish cases. First let $v = 2r$ be even.
Applying Lemma \ref{redto2}, in this case we may postcompose $y$ with an
automorphism of $\bold{X}$ such that $y = p^r\cdot\text{inc}_1$, i.e.
$\gamma$ is given by
$\gamma = (\gamma^0, \gamma^1) = (X(0), Y(0))$, with
\begin{equation*}
X(0) =
\begin{pmatrix}
p^r & 0\\
0   & 0
\end{pmatrix}\ ,\  Y (0) = \begin{pmatrix}
0 & 0\\
0 & p^r
\end{pmatrix}\ .
\end{equation*}
In order to lift $\gamma$ mod $t^p$, we search for matrices
$X(1), Y(1)\in {\rm M}_2 (A_p)$ such that
\begin{equation*}
\begin{aligned} X(1)\equiv X(0) \mbox{ in } A_1\\
                  Y(1)\equiv Y(0) \mbox{ in } A_1
\end{aligned}
\end{equation*}
and satisfying the identities
$$\sigma (X(1))\cdot S = U\cdot Y(1)\ ,\ \sigma (Y(1))\cdot S= \breve 
{U}\cdot Y(1)\ .$$
Here
\begin{equation*}
S=
\begin{pmatrix}
0 & 1 \\
p & 0
\end{pmatrix}\ .
\end{equation*}
Note that $\sigma$ can be viewed as a map $A_1\to A_p$. Since $\sigma (X(1)) = \sigma (X(0))$ and $\sigma (Y(1)) = \sigma (Y 
(0))$,
we obtain the identities
\begin{equation}
\begin{aligned}
\sigma (X(0))\cdot S &= U\cdot Y(1)
\\
\sigma (Y(0))\cdot S &= \breve{U}\cdot X(1)\ .
\end{aligned}
\end{equation}
Since $A_p$ has no $p$-torsion, we obtain as unique solution
\begin{equation*}
\begin{aligned}
Y(1)&= U^{-1}\cdot \sigma (X(0))\cdot S
\\
X(1)&= \breve{U}^{-1}\cdot \sigma (Y(0))\cdot S\ ,
\end{aligned}
\end{equation*}
provided that the matrices on the RHS have  coefficients which are  
integral,
i.e. which lie in $W[t]/t^p$. More precisely, we obtain $\ell\geq p$  
if these
coefficients are integral; otherwise $\ell$ is the maximum power $t^a$
such that the coefficients mod $t^a$ are integral.

Inserting the values for $X(0)$ and $Y(0)$, an easy calculation shows
\begin{equation}\label{indbegin}
X(1) = X(0)\quad\text{ and }\quad Y(1) = \left(\begin{matrix}
0 & -p^{r-1} t\cr
0 & p^r
\end{matrix}\right)\ .
\end{equation}
If follows that $\ell = 1$ if $r = 0$ and $\ell\geq p$ if $r\geq 1$.  
This proves the assertion for $v=0$.

In the next step we try to lift $\gamma$ from $A_p$ to $A_{p^2}$,
in the next step from $A_{p^2}$ to $A_{p^3}$ and inductively from
$A_{p^n}$ to $A_{p^{n+1}}$ for any $n\geq 1$. At each step we use the map $\sigma: A_{p^n}\to A_{p^{n+1}}$. This gives at each
step the recursive identities
\begin{equation}\label{recrel}
\begin{aligned}
Y(n+1) &= U^{-1}\cdot\sigma (X(n))\cdot S
\\
X(n+1) &= \breve{U}^{-1}\cdot \sigma (Y(n))\cdot S
\end{aligned}
\end{equation}
which can be solved after inverting $p$ for every $n$.

\medskip

\noindent{\bf Claim:} a) $X(2i) = X(2i+1)$ {\it and}  
$Y(2i+1)=Y(2i+2)$  for all $i\geq 0$.

  \noindent b) {\it There exist polynomials $P_0, P_1, \ldots, Q_0,  
Q_1, \ldots$ in $W[t]$ such
that}
\begin{equation*}
\begin{aligned}
X (2s) &= \begin{pmatrix}
\pm p^{r-s}\cdot t^{{p^{2s-1}+\ldots +p+1}}+p^{r-s+1}\cdot P_{2s} & 0\\
p^{r-s+1}\cdot Q_{2s} & 0
\end{pmatrix}
\\
Y (2s+1) &= \begin{pmatrix}
0 & \pm p^{r-s-1}\cdot t^{p^{2s+\ldots +p+1}}+p^{r-s}\cdot P_{2r+1}\cr
0 & p^{r-s}\cdot Q_{2r+1}
\end{pmatrix}\ .
\end{aligned}
\end{equation*}

\medskip

%\end{proof}
%\end{document}

\noindent Indeed, (\ref{indbegin}) shows this for the beginning terms  
with $P_0 = Q_0 = 0$ and
$P_1 = 0, Q_1 = 1$. The higher terms follow by induction from the  
recursive relations (\ref{recrel}).

The claim shows that $\gamma$ deforms to $A_{p^{2r}}$, but not to $A_ 
{p^{2r+1}}$.
In fact the upper right coefficient of $Y(2r+1)$ shows that $\gamma$  
deforms
precisely to $A_{p^{2r}+\ldots +p+1}$, which proves the proposition in
this case.

Next we consider the case when $v = 2r+1$ is odd. In this case, Lemma  
\ref{redto2} shows
that we may postcompose $y$ with an automorphism of $\bold{X}$ such that
$y = p^r\cdot\text{ inc}_2\circ\Pi$. Hence in this case
\begin{equation*}
X(0) = \left(\begin{matrix}
0 & 0\cr
p^{r+1} & 0
\end{matrix}\right)\quad ,\quad Y(0) = \left(\begin{matrix}
0 & p^r\cr
0 & 0
\end{matrix}\right)\ .
\end{equation*}
In this case an easy calculation using the identities (\ref{recrel})  
shows that
\begin{equation*}
X(1) = \left(\begin{matrix}
p^r t & 0\cr
p^{r+1} & 0
\end{matrix}\right)\quad ,\quad Y(1) = Y (0)\ .
\end{equation*}
Inductively one shows as before

\medskip

\noindent{\bf Claim:} a) $X(2i+1) = X(2i+2)$ {\it  
and } $Y(2i) = Y(2i+1)$ for all $i\geq 0$. 

\noindent b) {\it There exist polynomials $P_0, P_1, \ldots , Q_0,  
Q_1, \ldots$ in $W[t]$
such that}

\begin{equation*}
\begin{aligned}
X (2s+1) &= \begin{pmatrix}
\pm p^{r-s}\cdot t^{p^{2s}+\ldots +p+1}+p^{r-s+1}\cdot P_{2s+1} & 0\\
p^{r-s+1}\cdot Q_{2s+1} & 0
\end{pmatrix}
\\
Y (2s) &= \left(\begin{matrix}
0 & \pm p^{r-s}\cdot t^{p^{2s-1}+\ldots +p+1}+p^{r-s+1} P_{2s}\cr
0 & p^{r-s+1}\cdot Q_{2s}
\end{matrix}\right)\ .
\end{aligned}
\end{equation*}

\medskip

By looking at the upper right coefficient of $Y(2(r+1))$, we see that  
the deformation
locus of $\gamma$ is given by $t^{p^{2r+1}+\ldots +p+1} = 0$, which  
proves the
proposition in this case.

\end{proof}

%\end{document}

By Proposition~\ref{Zycompo},
\begin{equation}\label{expand}
\ZZ(y_1)\cdot \ZZ(y_2) = \sum_{\substack{s=0\\ \snass s\, \text 
{even}}}^{a}\ZZ_s\cdot \ZZ(y_2)
= \sum_{\substack{s=1\\ \snass s\, \text{\rm \ odd}}}^{b}\ZZ(y_1) 
\cdot \ZZ_s.
\end{equation}

%Recall that, for each $s\ge 0$, we have an identification $\Xs \tt  \F = \bold X = \bar\YY\times\YY$, where $\Xs = \OK\tt F_s$
%is the quasi canonical lift of $\bold X$ of level $s$.
Let $m_s(\yy)$ be the maximum $m$ such that the $\OK$-linear  
homomorphism
\begin{equation}\label{yymap}
  \bar\YY\times \bar\YY\  \overset{\mu(\yy)}{\lra}\ \bar \YY\times  
\YY=\bold X= \Xs\tt\FF
\end{equation}
lifts to a homomorphism
$$\bar Y\times \bar Y \dra \Xs$$
over $W_s/\pi_s^{m}W_s$, where $\mu(\yy)$ has matrix $\diag(\Pi^a, 
\Pi^b)$.
%Note that, if $s$ is even (resp. odd) a lift $\tilde y_1: \bar Y  \rightarrow \Xs$ of $y_1$
%(resp. $\tilde y_2: \bar Y \rightarrow \Xs$ of $y_2$)
%is provided on the divisor $Z_s$ by (\ref{Zslifts}) in the proof of  Proposition~\ref{Zycompo}.
Then
$$m_s(\yy) = \begin{cases} \ZZ_s\cdot \ZZ(y_2)&\text{for $s$ even,}\\
\nass
\ZZ(y_1)\cdot \ZZ_s&\text{for $s$ odd.}
\end{cases}
$$

We can write (\ref{yymap})
as
\begin{equation}\label{yymap2}
\Xo\tt\F = \bar\YY\times \YY \ \overset{\mu(\yy)}{\lra}\ \bar\YY 
\times \YY = \Xs\tt\FF\ ,
\end{equation}
where we have simply taken the conjugate linear $\OK$-action on the  
second factor
of the source $\bar \YY\times \bar\YY$. The matrix for $\mu(\yy)$ is  
unchanged,
and $m_s(\yy)$ is the maximum $m$ such that  this map lifts to a map
$$\Xo \dra \Xs$$
over $W_s/\pi_s^{m}W_s$.

To apply Theorem~\ref{tlinearThm}, we need to write $\mu(\yy)$ in the  
form (\ref{mumatrix}).
We take $1$ and $\delta$ as $\Z_p$-basis for $\OK$, and hence,  for  
any $p$-divisible group $X$, we have an identification
$\OK\tt X = X\times X$.
If $X$ is a $p$-divisible group with $\OK$-action, then
the isomorphism of Lemma~\ref{directsum}
$$X\times \bar X \isoarrow \OK\tt X=X\times X$$
has matrix
$$C=\begin{pmatrix} \delta&-\delta\\1&1\end{pmatrix}.$$
Thus, the matrix for $\mu(\yy)$ is
$$\frac12\begin{pmatrix}\Pi^a-\Pi^b& \delta(\Pi^b-\Pi^a)\\
(\Pi^b-\Pi^a)\delta^{-1}&\Pi^a+\Pi^b  \end{pmatrix} ,$$
if $s$ is even, and
$$\frac12\begin{pmatrix}(\Pi^b-\Pi^a)\delta^{-1}&\Pi^a+\Pi^b \\
  \Pi^a-\Pi^b& \delta(\Pi^b-\Pi^a)\end{pmatrix} ,$$
if $s$ is odd. Now if $s\le a$ is even, then $\Pi^a\in \Pi^s\OK$  so  
that  $l=b$ in Theorem~\ref{tlinearThm}.
If $s\le b$ is odd, then $\Pi^b\in \Pi^s\OK$ and $l=a$. This yields  
the following result.
\begin{prop}\label{maininter}
For $s\le a$ even,
$$\ZZ_s\cdot \ZZ(y_2) =  \begin{cases} {\disp \frac{p^{b+1}-1}{p-1}}& 
\text{if $b<s$,}\\
\nass
\nass
{\disp\frac{p^s-1}{p-1}}+ \frac12\,(b+1-s)\,e_s&\text{if $b\ge s$.}
\end{cases}
$$
For $s\le b$ odd,
$$\ZZ(y_1)\cdot \ZZ_s= \begin{cases} {\disp \frac{p^{a+1}-1}{p-1}}& 
\text{if $a<s$,}\\
\nass
\nass
{\disp\frac{p^s-1}{p-1}}+ \frac12\,(a+1-s)\,e_s&\text{if $a\ge s$.}
\end{cases}
$$
\end{prop}
Here recall that $e_s= p^{s-1}(p+1)$ for $s\ge 1$ and $e_0=1$. Also, if
$$\ZZ_s\cap \ZZ(y_2)=\Spec W_s/\pi_s^{\ell}\ ,$$ then ${\rm ord}_A 
(\pi_s^{\ell})=
(e/e_s)\cdot \ell$.

\begin{cor} Let $r\neq s$. Then
$$\ZZ_s\cdot \ZZ_t = e_{\min\{s,t\}} .$$
\end{cor}

By summing the $\ZZ_s\cdot \ZZ(y_2)$'s (resp. the $\ZZ(y_1)\cdot \ZZ_s 
$'s) of Proposition~\ref{maininter} as in (\ref{expand}), we obtain
the expression in Theorem~\ref{theomult}.

\newcommand{\vv}{\text{\bf v}}
\newcommand{\uu}{\text{\bf u}}

\section{Representation densities of hermitian forms}

In this section, we show that the expression given in Theorem~\ref 
{theomult} for the intersection multiplicity
in the case where the scaled fundamental matrix $\tilde T$ is $\GL_n(\OK)$- 
equivalent to $\diag(1_{n-2}, p^a,p^b)$
coincides, up to an elementary factor,  with the derivative of a  
certain representation density associated to $\tilde T$.
As explained in the introduction, this relation is the component at $p 
$ of an identity between
a global arithmetic intersection number or height, and a Fourier  
coefficient of the derivative
of an Eisenstein series on ${\rm U}(n,n)$. 
To avoid introducing additional notation, we continue to suppose that  
$\kay = \Q_{p^2}$ is the unramified
quadratic extension of $\Q_p$.

First recall that, for nonsingular matrices $S\in \Herm_m(\OK)$ and $T 
\in \Herm_n(\OK)$, with $m\ge n$,
the representation density $\a_p(S,T)$ is defined as 
 \begin{equation}
\a_p(S,T) = \lim_{k\to\infty} (p^{-k})^{n(2m-n))} |A_{p^k}(S,T)|,
\end{equation}
 where
$$A_{p^k}(S,T) =\{\ x\in M_{m,n}(\OK/p^k\OK)\mid S[x]\equiv T\mod p^k \ \},$$
where $S[x] = {}^tx S\s(x)$.
The density depends only on the $\GL_m(\OK)$- (resp. $\GL_n(\OK)$-)  
equivalence class of $S$ (resp. $T$).
An explicit formula for $\a_p(S,T)$ has been given by Hironaka, \cite 
{hironaka}.

Let $\ell(T)$ be the smallest $\ell$ such that $p^{\ell}T^{-1}\in  
\Herm_n(\OK)$.
In fact, for $k>\ell(T)$, the quantity $(p^{-k})^{n(2m-n))} |A_{p^k} 
(S,T)|$ is constant  and is non-zero
if and only if there exists an $x\in M_{m,n}(\OK)$ such that $S[x]=T$.
For $r\ge 0$, let $S_r = \diag(S,1_{r})$. Then
$$\a_p(S_r,T) = F_p(S,T; (-p)^{-r})$$
for a polynomial $F_p(S,T;X)\in \Q[X]$, as can be seen immediately
from Hironaka's formula.

Recall that the isometry class of a non-degenerate hermitian space $V 
$ of dimension $n$ over $\kay$ is determined by
its determinant
$\det(V)\in \Q_p^\times/N(\kay^\times)$. Thus, if $S$ and $T\in  
\Herm_n(\OK)$ are non-singular
with $\ord(\det(S))+\ord(\det(T))$ odd, then
$\a_p(S,T)=0$. In this case, we define the derivative of the  
representation density
$$\a'_p(S,T) = -\frac{\d}{\d X}F_p(S,T;X)\vert_{X=1}.$$

The main result of this section is the following.
\begin{prop}\label{dender}
Let $S= 1_n$ and $T=\diag(1_{n-2}, p^a, p^b)$ for $0\le a< b$ with  
$a+b$ odd.
Then $\a_p(S,T)=0$ and
$$\frac{\a_p'(S,T)}{\a_p(S,S)} =  \frac12\sum_{\ell=0}^{a} p^\ell (a 
+b-2\ell+1),$$
where
$$\a_p(S,S) =  \prod_{\ell=1}^{n}(1- (-1)^\ell p^{-\ell}).$$
\end{prop}

Comparing this expression with that given in Theorem~\ref{theomult},  
we find the following relation
between the derivative of the hermitian representation density and  
the arithmetic intersection multiplicity.

\begin{cor} Let $\ZZ_{i,j}(\xx)$ be non-empty, with associated scaled fundamental
 matrix $\tilde T=p^{2i-j}h(\xx,\xx)\in \Herm_n(O_\kay)$. 
Suppose that $\tilde T$ is $\GL_n(\OK)$-equivalent to $ 
\diag(1_{n-2}, p^a, p^b)$  with $a+b$ odd.
Then
$$\degh(\ZZ_{i,j}(\xx)) = \log(p)\cdot \frac{\a_p'(S,\tilde T)}{\a_p 
(S,S)} .$$
\end{cor}

\begin{proof}[Proof of Proposition~\ref{dender}]
The first step is the following reduction formula, which is the  
hermitian analogue of
Corollary~5.6.1 in \cite{kitaoka}. For the convenience of the reader,  
we will sketch the proof below.
\begin{prop}\label{reduction}
Let $S'= 1_2$ and  $T'=\diag(p^a,p^b)$. Then
$$\a_p(S_r,T) = \a_p(S_r, 1_{n-2}) \, \a_p(S'_{r},T').$$
\end{prop}

It follows that
$$\a'_p(S,T) = \a_p(S, 1_{n-2})\,\a'_p(S',T').$$
By Hironaka's formula or the classical literature, \cite{shimura},
$$F_p(1_n,1_n; X) = \prod_{\ell=1}^{n}(1- (-1)^\ell p^{-\ell} X),$$
so that
$$\a_p(1_n,1_{n-2}) = \prod_{\ell=1}^{n-2}(1- (-1)^\ell p^{-\ell-2}).$$
On the other hand, Nagaoka, \cite{nagaoka}, proved the following in  
the binary case.
\begin{prop}[Nagaoka] Suppose that $S'=1_2$ and that $T'= \diag 
(p^a,p^b)$ with $0\le a\le b$, but with no
condition on the parity of $a+b$. Then
$$F_p(S',T';X) = (1+p^{-1}X)(1-p^{-2}X) \sum_{\ell=0}^a (pX)^\ell\bigg 
(\sum_{k=0}^{a+b-2\ell}(-X)^k\bigg).$$
\end{prop}
\begin{cor} If $a+b$ is odd, then
$$(1-p^{-2})^{-1}(1+p^{-1})^{-1}\a'_p(S',T') = \frac12\sum_{\ell=0}^ 
{a} p^\ell\ (a+b-2\ell+1).$$
\end{cor}

This completes the proof of Proposition~\ref{dender}.
\end{proof}

\begin{proof}[Proof of Proposition~\ref{reduction}]

The proof is just the hermitian version of the argument given by  
Kitaoka, \cite{kitaoka}, pp. 104--107.

First we pass to the lattice formulation in the standard way.
Viewing $S$  as the matrix of inner products $((v_i,v_j))$
for a basis $\vv=[v_1,\dots,v_m]$ of an $\OK$-lattice $M$ and $T$ as  
the matrix of inner products
$((u_i,u_j))$ for a basis $\uu=[u_1,\dots,u_n]$ for an $\OK$-lattice  
$L$, we have a bijection of $A_{p^r}(S,T)$
with the set
\begin{multline}
I_k(L,M) = \{\ \ph\in \Hom_{\OK}(L,M/p^kM)\mid \\
\nass
  (\ph(x),\ph(y))\equiv (x,y)\mod p^k, \ \forall x, y\in L\ \}.
  \end{multline}
Then
$$\a_p(S,T) = \a_p(L,M) = (p^{-k})^{n(2m-n)}|I_k(L,M)|,$$
for $k$ sufficiently large.

We need the following preliminary results.

\begin{lem} Suppose that $N\subset M$ is a regular\footnote{Here, following the terminology in 
\cite{kitaoka},  regular means that the restriction of the hermitian form 
to $N$ is non-degenerate.} sublattice with $ 
\OK$-basis $\{v_i\}$, so that 
$N=[v_1,\dots,v_r]$. Suppose that $w_i\in M$ is sufficiently close to  
$v_i$.
Then there is an isometry $\eta\in {\rm U}(M)$ with $\eta(N) = [w_1, 
\dots,w_r]$.
\end{lem}

\begin{lem} If $\ph:L\rightarrow M$ with $(\ph(x),\ph(y))\equiv (x,y) 
\mod p^k$, for some sufficiently large $k$, then there is an isometry $\eta: L \rightarrow \ph 
(L)\subset M$.
\end{lem}

\begin{lem} Suppose that $\ph_1$ and $\ph_2$ are two homomorphisms  
satisfying the
conditions of the previous lemma. Also suppose that
$\ph_1\equiv \ph_2\mod p^k$, for some sufficiently large $k$.
Then there is an isometry $\gamma\in {\rm U}(M)$
such that $\ph_2(L)=\gamma(\ph_1(L))$.
\end{lem}

For given $L$ and $M$ and a sublattice $N\subset M$ such that $N$ is  
isometric to $L$, we 
let 
$$\tilde{I}_k(L,M) = \{\ \ph\in \Hom_{\OK}(L,M)\mid (\ph(x),\ph(y))\equiv (x,y) 
\mod p^k, \ \forall x, \ y \in L\ \},$$
and define
$$\tilde{I}_k(L,M;N) = \{\ \ph\in \tilde{I}_k(L,M)\mid \exists\, \eta\in {\rm U} 
(M)\ \text{with}\ \ph(L) = \eta(N)\ \}$$
and
$$I_k(L,M;N) = \{\ \ph\in I_k(L,M)\mid\exists\,  \eta\in {\rm U}(M)\ \text{with} 
\ \tilde\ph(L) = \eta(N)\ \}.$$
In this last set $\tilde\ph\in \tilde{I}_k(L,M)$ is a preimage of $\ph$.
By the preliminary lemmas, these
sets are well defined for $k$ sufficiently large.

\begin{prop} Suppose that $L= L_1\perp L_2$ with $L_j$ of rank $n_j$.
Let $\{N_i\}$ be a set of representatives for the ${\rm U}(M)$-orbits  
in the set of
all sublattices $N\subset M$ such that $N$ is isometric to $L_1$. Then
\begin{multline}
|I_k(L,M)|= \sum_i |I_k(L_1,M;N_i)|\\
\nass
\times|\{\ph_2\in I_k(L_2,M)\mid (\ph_2(L_2),N_i)\equiv 0\mod p^k\,\}|.
\end{multline}
\end{prop}
\begin{proof}
First note that, for $k$ sufficiently large,
for any $\ph_1\in \tilde{I}_k(L_1,M)$, $\ph_1(L_1)$ is isometric to  
$L_1$.
For each $\ph_1\in \tilde{I}_k(L_1,M)$ choose an isometry $\gamma= 
\gamma(\ph_1)\in {\rm U}(M)$
such that $N_i = \gamma (\ph_1(L_1))$, for some $i$.
Also, for each $\ph_1\in I_k(L_1,M)$, choose a preimage $\tilde\ph_1 
\in \tilde{I}_k(L_1,M)$.
There is then a bijection
$$I_k(L,M) \isoarrow \coprod_i I_k(L_1,M;N_i)\times
\{\ph_2\in I_k(L_2,M)\mid (\ph_2(L_2),N_i) \equiv 0\mod p^k\}$$
given by $\ph \mapsto (\ph_1, \ph_2)$ with $\ph_1=\ph\vert_{L_1}$
and $\ph_2= \gamma(\tilde\ph_1)\circ \ph\vert_{L_2}$.
\end{proof}

\begin{lem}
\begin{multline*}
|\{\ph_2\in I_k(L_2,M)\mid (\ph_2(L_2),N_i)\equiv 0\mod p^k\ \}|\\
\nass
= |N_i^\vee:N_i|^{n_2}|M:N_i\perp N_i^\perp|^{-n_2} |I_k(L_2,N_i^ 
\perp)|.
\end{multline*}
\end{lem}
\begin{proof}
As in Kitaoka,
$$\{x\in M\mid (x,N_i)\equiv 0\mod p^k\} = p^kN_i^\vee\perp N_i^\perp.$$
where $N_i^\perp = (\kay N_i)^\perp\cap N$, provided $p^kN_i^\vee 
\subset N_i$.
Thus
\begin{multline}
\{\ph_2\in I_k(L_2,M)\mid (\ph_2(L_2),N_i) \equiv 0\mod p^k\}\\
\nass
  =
\{\ph_2: L_2\rightarrow p^kN_i^\vee\perp N_i^\perp \mod p^k M \mid  
(\ph_2(x),\ph_2(y))\equiv (x,y)\mod p^k\,\}.
\end{multline}
Next, we replace $p^kM$ by $p^k(p^aN_i^\vee\perp N_i^\perp)$, so that  
the cosets diagonalize,
i.e., we consider the set
\begin{multline}\label{tempph2}
\{\ph_2: L_2\rightarrow p^kN_i^\vee\perp N_i^\perp \mod p^k(p^aN_i^ 
\vee\perp N_i^\perp)\mid \\
\nass
(\ph_2(x),\ph_2(y))\equiv (x,y)\mod p^k\, \}.
\end{multline}
Write $\ph_2=\psi_1+\psi_2$ with $\psi_1:L_2\rightarrow p^kN_i^\vee/p^ 
{k+a}N_i^\vee$ and
$\psi_2: L_2\rightarrow N_i^\perp/p^kN_i^\perp$. Since we are  
assuming that
$p^kN_i^\vee\subset N_i$,  we have
$(\psi_1(x),\psi_1(y))\in (N_i,p^kN_i^\vee)\subset p^k\OK$. Thus the  
condition on $\ph_2$ in
(\ref{tempph2}) just amounts to the condition
$(\psi_2(x),\psi_2(y))\equiv (x,y)\mod p^k$, with no restriction on $ 
\psi_1\in\Hom_{\OK}(L_2, p^kN_i^\vee/p^{k+a}N_i^\vee)$.
This yields the claimed expression, once the various lattice indices  
are taken into account.
\end{proof}

\begin{cor}  With the notation of the previous proposition,
$$\a_p(L,M) = \sum_i |M:N_i\perp N_i^\perp|^{-n_2}|N_i^\vee:N_i|^{n_2} 
\,\a_p(L_1,M;N_i)\,\a_p(L_2,N_i^\perp).$$
\end{cor}

Now suppose that $L_1$ is unimodular, so that any $N\subset M$  
isometric to $L_1$ is unimodular.
Then, for any such $N$, $M=N\perp N^\perp$. Moreover, since $\kay$ is  
unramified, if
$$M= N_1\perp N_1^\perp = N_2\perp N_2^\perp$$
are two such decompositions, then $N_1^\perp \simeq N_2^\perp$. Thus $ 
{\rm U}(M)$ acts transitively
on such $N$'s.

\begin{cor}  With the notation of the previous proposition, suppose  
that $L_1$ is unimodular. Then
$$\a_p(L,M) = \a_p(L_1,M)\,\a_p(L_2,N^\perp).$$
\end{cor}

This proves Proposition~\ref{reduction}. \end{proof}

{\bf Correction to \cite{KRYbook} and \cite{rapo}.}

We take this occasion to close a gap in \cite{KRYbook}, where we  
inadvertently
omitted the proof of Lemma 7.7.3. This lemma is  also 
implicitly used in
\cite{rapo} (the equality of divisors right after Lemma 3.1). We  
formulate here the lemma and give the proof. 
\begin{lem}
Let $G$ be the formal $p$-divisible group of dimension $1$ and height  
$2$ over
$\FF$. Let $\mathcal M=\Spec W [[t]]$ be the universal deformation  
space of $G$. Let
$\varphi\in\End (G)$ be an endomorphism which generates an order of  
conductor $c$
in a quadratic extension $\kay$ of $\Q_p$. Let $\mathcal T =\mathcal  
T (\varphi)$ be the deformation
locus of $\varphi$ (a relative divisor on $\mathcal M$, by \cite 
{rapo}, Prop. 1.4).
Then there is an equality of divisors on $\mathcal M$,
\begin{equation*}
\mathcal T = \sum_{s=0}^{c}\  \mathcal W_s (\varphi)\ ,
\end{equation*}
where $\mathcal W_s(\varphi)$ denotes the quasi-canonical divisor of  
level $s$ (relative
to $\kay$).
\end{lem}
\begin{proof}
All quasi-canonical divisors $\mathcal W_s (\varphi)$ are prime  
divisors which are
pairwise distinct and with $\mathcal W_s (\varphi) \subset\mathcal T$  
for $0\leq s\leq c$.
Hence we have an inequality of relative divisors
\begin{equation*}
\sum_{s=0}^{c}\  \mathcal W_s (\varphi)\leq\mathcal T\ .
\end{equation*}
In order to show equality here, it suffices to compare the intersection
multiplicities with the special fiber $\mathcal M_p = \Spec\FF [[t]] 
$. For the LHS,
this is equal to
\begin{equation}\label{altern}
\sum_{s=0}^{c} e_s = \
\begin{cases}\
2\cdot\sum_{i=0}^{c-1} p^i+p^c & \kay/\Q_p\ \text{ unramified}\\
\ 2\cdot\sum_{i=0}^{c} p^i & \kay/\Q_p\ \text{ ramified}\  .
\end{cases}
\end{equation}

Here $e_s = [W_s : W]$ denotes the absolute ramification index.

To determine $\mathcal M_p\cdot\mathcal T$,  we first note that 
$\varphi\in (\Z_p+\Pi^\ell O_D)\setminus (\Z_p+\Pi^{\ell+1} O_D)$, where $\ell=2c$ in case $\kay$ is unramified over $\Q_p$, and $\ell=2c+1$ in case $\kay$ is ramified over $\Q_p$, comp. \cite{wewers2}, 1.2. Now we use the result of Keating \cite{keating}, Theorem~1.1 (see also \cite{vollaardARGOS}, Theorem~2.1),  which gives as the length of the  deformation locus of $\varphi$ in $\mathcal M_p$  exactly  the expression on the RHS of equation (\ref{altern}) above. 
\end{proof}

%For the RHS we use \cite{rapo}, Prop. 1.8. As the proof of that \marginpar {has to be modified}
%proposition shows, this gives a formula for $(\mathcal T_1 \cdot  
%\mathcal T_2 \cdot \mathcal S_p)$
%where $\mathcal S_p$ denotes the special fiber of the universal  
%deformation
%space of $(G, G)$ and where $\mathcal T_1 = \mathcal T (\psi_1)$ and $ 
%\mathcal T_2 = \mathcal T (\psi_2)$
%are the deformation loci of two isogenies $\psi_1, \psi_2 : G\to G$
%{\it with the property that $\psi_2$ has maximal valuation in its  
%residue
%class modulo $\Z_p\cdot\psi_1$.} We apply this formula in the case where
%$\psi_1$ is an isomorphism and where $^t\psi_1\circ\psi_2$ is our given
%endomorphism $\varphi$. More precisely, we change $\varphi$ by an  
%element in
%$\Z_p$ so that it satisfies the maximality condition above. This change
%does not affect $\mathcal T$. Let $a = v (\varphi)$ be the valuation  
%of $\varphi$. Then
%by \cite{wewers2}, 1.2., we have $c = \left[\frac{a}{2}\right]$.
%Furthermore, by \cite{rapo}, Prop. 1.8.
%\begin{equation*}
%(\mathcal T\cdot \mathcal M_p) = \
%\begin{cases}\
%2\cdot\sum_{i=0}^{c-1} p^i+p^c & a\ \text{ even}\\
%\ 2\cdot\sum_{i=0}^{c} p^i & a\ \text{ odd}\ .
%\end{cases}
%\end{equation*}
%Hence it remains to see that $\kay/\Q_p$ is unramified if and only if  
%$a$
%is even. This follows from \cite{wewers2}, 1.2.

\end{document}